\documentclass[11pt, oneside]{amsart}
\usepackage{pb-diagram,amssymb,epic,eepic,verbatim,braket, xcolor}
\usepackage{wasysym,graphics,epsfig,psfrag,braket,marginnote}
\usepackage{hyperref,url}
\hypersetup{colorlinks}
\usepackage{color}
\usepackage[numbers]{natbib}

\usepackage[shortlabels]{enumitem}
\usepackage{soul}
\usepackage{bm}
\newcommand{\beqn}{\begin{equation*}}
\newcommand{\eeqn}{\end{equation*}}
\newcommand{\beq}{\begin{equation}}
\newcommand{\eeq}{\end{equation}}

\newcommand{\mhat}{\hat}

\usepackage{xr}
\usepackage{amsfonts,transparent}
\DeclareMathAlphabet{\mathpgoth}{OT1}{pgoth}{m}{n}
\DeclareMathAlphabet{\mathesstixfrak}{U}{esstixfrak}{m}{n}
\DeclareMathAlphabet{\mathboondoxfrak}{U}{BOONDOX-frak}{m}{n}
\usepackage{amsmath,amssymb}
\usepackage[bbgreekl]{mathbbol}
\usepackage{yfonts,mathtools}
\definecolor{darkred}{rgb}{0.5,0,0}
\definecolor{darkgreen}{rgb}{0,0.5,0}
\definecolor{darkblue}{rgb}{0,0,0.5}
\hypersetup{ colorlinks,
linkcolor=darkblue,
filecolor=darkgreen,
urlcolor=darkred,
citecolor=darkblue }
\newtheorem{theorem}{Theorem}[section]

\newtheorem{corollary}[theorem]{Corollary}

\newtheorem{proposition}[theorem]{Proposition}
\newtheorem{lemma}[theorem]{Lemma}
\newtheorem{lem}[theorem]{}
\theoremstyle{definition}
\newtheorem{definition}[theorem]{Definition}
\theoremstyle{remark}
\newtheorem{remark}[theorem]{Remark}
\newtheorem{example}[theorem]{Example}
\newcommand{\blem}{\begin{lem} \rm}
\newcommand{\elem}{\end{lem}}

\newcommand\M{\mathcal{M}}

\renewcommand\M{\mathcal{M}}

\newcommand\XX{\mathbb{X}}

\newcommand{\R}{\mathbb{R}}

\newcommand{\C}{\mathbb{C}}
\newcommand{\CP}{\mathbb{C}P}

\newcommand{\cW}{\mathcal{W}}

\newcommand{\Z}{\mathbb{Z}}

\renewcommand{\P}{\mathbb{P}}
\newcommand{\vv}{\check}

\newcommand{\on}{\operatorname}

\newcommand\white{{\includegraphics[width=.05in]{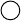}}}

\newcommand\black{{\includegraphics[width=.05in]{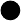}}}

\newcommand{\ainfty}{{$A_\infty$\ }}
\newcommand{\ab}{\ab}

\newcommand{\Edge}{\on{Edge}}

\newcommand{\Aug}{\on{Aug}}
\newcommand{\Cliff}{\on{Cliff}}
\newcommand{\Hopf}{\on{Hopf}}

\newcommand{\Ver}{\on{Vert}}

\renewcommand{\ker}{ \on{ker}}

\newcommand{\ann}{{\on{ann}}}

\newcommand\dirac{/\kern-1.2ex\partial} 
\newcommand\qu{/\kern-.7ex/} 
\newcommand\lqu{\backslash \kern-.7ex \backslash} 

\newcommand\dr{r_+ \kern-.7ex - \kern-.7ex r_-}
 
\newcommand{\labell}\label

\renewcommand{\d}{{\on{d}}}
\newcommand{\ol}{\overline}

\newcommand\eps{\epsilon}

\newcommand{\ti}{\tilde}

\newcommand\cR{\mathcal{R}}

\newcommand\cI{\mathcal{I}}

\newcommand\Rep{\on{Rep}}

\newcommand\ev{\on{ev}}

\newcommand\Vect{\on{Vect}}
\renewcommand\ul{\underline}

\newcommand\grad{\on{grad}}

\newcommand\bdefn{\begin{definition}}
\newcommand\edefn{\end{definition}}
\newcommand\bea{\begin{eqnarray*}}
\newcommand\eea{\end{eqnarray*}}
\newcommand\bcv{\left[ \begin{array}{r} }
\newcommand\ecv{\end{array} \right] }

\newcommand\bma{\left[ \begin{array}{l} }
\newcommand\ema{\end{array} \right]}
\newcommand\ben{\begin{enumerate}}
\newcommand\een{\end{enumerate}}
\newcommand\bex{\begin{example}}
\newcommand\bsj{\left\{ \begin{array}{rrr} }
\newcommand\esj{\end{array} \right\}}

\newcommand{\Newt}{\on{Newt}}
\newcommand\eex{\end{example}}

\newcommand\sx{*\kern-.5ex_X}

\def\mathunderaccent#1{\let\theaccent#1\mathpalette\putaccentunder}
\def\putaccentunder#1#2{\oalign{$#1#2$\crcr\hidewidth \vbox
to.2ex{\hbox{$#1\theaccent{}$}\vss}\hidewidth}}

\renewcommand{\aa}{\mathfrak{a}}
\renewcommand{\ab}{\on{ab}}
\newcommand{\bb}{\mathfrak{b}}
\renewcommand{\aa}{\mathfrak{a}}
\newcommand{\cc}{\mathfrak{c}}
\renewcommand{\ss}{\mathfrak{s}}

\newcommand\cwo[1]{ { \color{darkred}  } }

\makeatletter
\newcommand\bigDiamond{\mathop{\mathpalette\bigDi@mond\relax}}
\newcommand\bigDi@mond[2]{%
  \vcenter{\hbox{\m@th
    \scalebox{\ifx#1\displaystyle 2\else1.2\fi}{$#1\Diamond$}%
  }}%
}
\newcommand\bigLozenge{\mathop{\mathpalette\bigL@zenge\relax}}
\newcommand\bigL@zenge[2]{%
  \vcenter{\hbox{\m@th
    \scalebox{\ifx#1\displaystyle 2\else1.2\fi}{$#1\blacklozenge$}%
  }}%
}
\makeatother

\definecolor{darkred}{rgb}{0.5,0,0}
\definecolor{darkpurple}{rgb}{0.5,0,.5}
\definecolor{darkpink}{rgb}{0,.5,0.5}
\definecolor{cyan}{rgb}{.25,0,0.75}
\definecolor{darkgreen}{rgb}{0,0.5,0}
\definecolor{darkblue}{rgb}{0,0,0.5}

\begin{document}

\title[Augmentation varieties and disk potentials ]{Augmentation varieties and disk potentials III}


\author{Kenneth Blakey}
\address{Department of Mathematics,182 Memorial Drive, Cambridge, MA 02139, U.S.A.} \email{kblakey@mit.edu}

\author{Soham Chanda}
\address{Mathematics-Hill Center, Rutgers University, 110
  Frelinghuysen Road, Piscataway, NJ 08854-8019, U.S.A.  } \email{sc1929@math.rutgers.edu}

\author{Yuhan Sun}
\address{Huxley Building,
South Kensington Campus,
Imperial College London,
London,
SW7 2AZ, U.K}\email{yuhan.sun@imperial.ac.uk}

  \author{Chris Woodward}
\address{Mathematics-Hill Center, Rutgers University, 110
  Frelinghuysen Road, Piscataway, NJ 08854-8019, U.S.A.}
\email{ctw@math.rutgers.edu}

\thanks{Chanda and Woodward were partially supported by NSF grant DMS 
  2105417 and Sun was partially supported by the
EPSRC grant EP/W015889/1.
 Any opinions, findings, and conclusions or recommendations 
  expressed in this material are those of the author(s) and do not 
  necessarily reflect the views of the National Science Foundation.  }

\begin{abstract}  This is the third in a series of papers in which 
we construct  Chekanov-Eliashberg algebras for Legendrians in circle-fibered
contact manifolds and study the associated augmentation varieties.  In 
this part, we prove that for connected Legendrian covers of monotone Lagrangian tori, the augmentation variety is equal to the image of the zero level set of the disk potential, as suggested by  Dimitroglou-Rizell-Golovko \cite{dr:bs}.  In particular, 
we show that Legendrian lifts of Vianna's exotic tori are not Legendrian isotopic,  as conjectured in \cite{dr:bs}.   Using related ideas, we show that the Legendrian lift of the Clifford torus admits no exact fillings, extending results of  Dimitroglou-Rizell \cite{dr:few} and Treumann-Zaslow \cite{treumann:cubic} in dimension two.  We consider certain disconnected Legendrians, and show, similar to another suggestion of  Aganagic-Ekholm-Ng-Vafa \cite{aganagic}, that the components of the augmentation variety correspond to certain partitions and each component is defined by a (not necessarily exact) Lagrangian filling. 
\end{abstract}

\maketitle

\tableofcontents

\section{Introduction} 

In this third part in the series, we construct an 
augmentation variety associated to Legendrians in circle-fibered contact
manifolds, and prove that it is a Legendrian isotopy invariant.   The augmentation variety of a Legendrian $\Lambda$ is a subset of the representation variety 
\[ \Aug(\Lambda ) \subset \Rep(\Lambda) \]
defined by the ideal in the Chekanov-Eliashberg algebra that is the union of 
kernels of augmentations.  It  is closely  related to the space of fillings of the Legendrian \cite{ref1,ref2,ref3,ref4,ref5,ref6,ref7,riz:lift,ref9,vafaetal,ref11,ref12,ref13,ref14,ref15}
and as such plays a role in the possible desingularizations of a singular Lagrangian in a symplectic manifold.  In particular, we will show that any tamed filling 
defines a subset of the augmentation variety defined by polynomials
that are killed by the corresponding augmentation.    We denote by 
\[ \Aug_{\on{geom}}(\Lambda) \subset \Aug(\Lambda) \]
the locus of the augmentation variety of a Legendrian $\Lambda$ defined by tamed fillings; it is also a Legendrian isotopy invariant.  In examples, each tamed filling $L$ defines an irreducible component 
\[ \Aug_L(\Lambda) \subset \Aug_{\on{geom}}(\Lambda) \subset \Aug(\Lambda) ;\]
it would be interesting to know whether this is always the case.  
In particular, the Harvey-Lawson Lagrangian considered
as a filling has augmentation variety equal to that of the Legendrian.
The same phenomenon repeats itself for Legendrians associated to other Fano 
toric varieties. 
The study of fillings and augmentation varieties is related to the following question in mirror symmetry.
The mirror symmetry conjectures suggest that the space of deformations of a  Lagrangian brane should have 
 a complex analytic structure, and examples show that such a deformation space is only well-behaved if one includes certain surgery operations.  In this surgery operation, 
the Lagrangian is modified locally by replacing one Lagrangian filling of a Legendrian by 
another.  Thus the space of fillings describes the possible surgery operations.

%
The main result is  a relationship between
augmentation varieties and potentials suggested by 
Dimitroglou-Rizell-Golovko in \cite[Conjecture 9.1]{dr:bs}.  
Let 
\[ \Pi \subset Y \] 
be a compact, connected, monotone, relatively spin Lagrangian.  The disk potential is a polynomial function 
\[ W_\Pi: \Rep(\Pi) \to \C \]
defined by a count of Maslov two holomorphic disks passing through a generic point in $\Pi$.  In the language of \ainfty algebras, it
has the following interpretation:  Let 
\[ m_d: CF(\Pi)^{\otimes d} \to CF(\Pi) \]
be the structure maps of the Fukaya algebra $CF(\Pi)$ of the Lagrangian $\Pi \subset Y$, defined 
for simplicity over $\C$. The zero-th structure map 
\[ m_0: CF(\Pi)^{\otimes 0} := \C \to CF(\Pi) \] 
has image $m_0(1) \in CF(\Pi)$ the {\em curvature} of the Fukaya algebra $CF(\Pi)$.  The projective version of the Maurer-Cartan equation
requires that $m_0(1)$ is a multiple $w$ of the unit $1_\Pi \in
CF(\Pi)$, in which case one says that $CF(\Pi)$ is projectively flat.  We say that a projectively flat $CF(\Pi)$ is {\em flat} if $w$ vanishes.  
Under suitable monotonicity assumptions $CF(\Pi)$ is automatically
projective flat and  viewing $w(y)$ as a function of the
local system $y$ on $\Pi$ defines the potential
$ W_\Pi : \Rep(\Pi) \to \C .$  
The zero level
set of the potential is then the space of absolute, rather than
projective, Maurer-Cartan solutions.  

In the version of Legendrian contact homology considered
in this paper, we have the following relationship between the disk potential of the Lagrangian projection and the augmentation variety.
Such a relationship was conjectured in Dimitroglou-Rizell-Golovko \cite{dr:bs} in the Ekholm-Entyre-Sullivan model \cite{ees:leg} on the basis of computations for  the Clifford and Chekanov tori.  Let $Z$ be a negatively curved circle bundle over $Y$ and let $\Lambda \subset Z$ be a horizontal lift of $\Pi$ and 
denote the map on representation varieties induced by $p$ 
\[ \Rep(p): \Rep(\Pi) \to \Rep(\Lambda) .\]


\begin{theorem} \label{augpot}  
Suppose that $(Z,\Lambda)$ is a tame pair and the projection $\Pi \subset Y$ 
is an embedded submanifold. The augmentation variety $\Aug(\Lambda)$ is equal to the image $\Rep(p)(
W_\Pi^{-1}(0))$ of the zero level set of the disk potential $W_\Pi$.
  \end{theorem}

One fundamental example for us is a horizontal lift of the Clifford torus
\begin{equation} \label{picliff} 
\Pi_{\Cliff} = \Set{ [ z_1 : \ldots : z_n ] \ | \ |z_i|^2 = |z_j|^2 \text{ for all } i,j } \cong (S^1)^{n-1}
\subset \C P^{n-1} 
\end{equation}
in complex projective space $Y = \CP^{n-1}$. \label{rep:horizlift}
A horizontal lift of the Clifford torus to the unit circle bundle of the tautological line bundle is a Legendrian torus in the standard contact sphere $Z = S^{2n-1}$: 
\begin{equation} \label{cliffleg}    \Lambda_{\Cliff} = \Set{ (z_1,\ldots, z_n) \in \C^n |  \begin{array}{l}
                                                     |z_1|^2 = \ldots
                                                     = |z_n|^2 = \frac{1}{n} \\  z_1 
    \ldots z_n \in (0,\infty)  \end{array} } \cong (S^1)^{n-1} \subset 
  S^{2n-1} .\end{equation}
%
The other lifts are obtained by the $S^1$ action on $S^{2n-1}.$ The map
\[ \C^n  - \{ 0 \} \to \C P^{n-1}, \quad z \mapsto \on{span}(z) \] 
restricts an $n$-fold cover of $\Pi_{\Cliff}$.  

For example, in the case $\Pi \subset \CP^2$ is the Clifford torus then one obtains 
\[ W_\Pi: \Rep(\Pi) \to \C, \quad (\hat{y}_1, \hat{y}_2) \mapsto 
- \hat{y}_1 - \hat{y}_2 + \hat{y}_1^{-1} \hat{y}_2^{-1}  \]
in coordinates $\hat{y}_1,\hat{y}_2$ on $\Rep(\Pi) \cong T^2$, with each term 
corresponding to one of the three Maslov-index-two disks in the
complex projective plan bounding the Clifford torus.   We have 
in terms of coordinates $y_1,y_2$ on $\Rep(\Lambda)$
\[ - \hat{y}_1 - \hat{y}_2 + \hat{y}_1^{-1} \hat{y}_2^{-1}  = 
( - \hat{y}_1^2 \hat{y}_2 - \hat{y}_1 \hat{y}_2^2 + 1 )
 \hat{y}_1^{-1} \hat{y}_2^{-1}  = (- y_1 - y_2 + 1)\hat{y}_1^{-1} \hat{y}_2^{-1} . \] 
 It follows that the image of the zero-level set (by the Theorem,
 equal to the augmentation variety) is 
 \[ \Rep(p)(
W_\Pi^{-1}(0)) = \{ y_1 + y_2 = 1 \} \subset \Rep(\Lambda) .\]

As a corollary of Theorem \ref{augpot} we show the existence of infinitely many Legendrian tori in odd dimensional spheres with the standard contact structure which are pair-wise non-isotopic. In the case of the lifts of Vianna tori and the tori in Chanda-Hirschi-Wang \cite{chw:tori}, invariance of the augmentation variety under Legendrian isotopy implies that these Legendrian tori are non-isotopic. 

\begin{corollary}(Corollary \ref{lotsofleg} below)  For $n > 1$, an odd-dimensional sphere with standard contact structure $(S^{2n-1},\xi_{std})$ has infinitely many Legendrian tori which are not Legendrian isotopic to each other.
 \end{corollary}

The examples of
disconnected Legendrians in Theorem \ref{partitions} show that for
disconnected Legendrians the augmentation variety may be reducible,
and so not directly related to the Maurer-Cartan space for the Lagrangian projection. In some cases, we may also compute the geometric
augmentation variety:

\begin{theorem}  Suppose that $Y$ is a monotone toric variety  and $\Pi$ is a monotone Lagrangian torus orbit with a non-trivial spin structure.
Then  the geometric augmentation variety $\Aug_{\on{geom}}(\Lambda)$ is equal to $\Aug(\Lambda)$.
\end{theorem}
      
By studying the topology of the moduli spaces of pseudoholomorphic maps, 
we generalize a result of
Dimitroglou-Rizell \cite{dr:few} and Treumann-Zaslow
\cite{treumann:cubic} ruling out exact fillings of the Clifford Legendrian, see Dimitroglou-Rizell-Golovko
\cite[p.3]{dr:bs}, to arbitrary dimension:

\begin{theorem} (Theorem \ref{nofill} below) The Clifford
  Legendrian $\Lambda_{\Cliff} \cong T^{n-1} \subset S^{2n-1}$ has no exact Lagrangian
  filling for $n > 2$.  It has no spin Lagrangian filling for the
  trivial spin structure, and for any non-trivial spin structure its
  augmentation variety $\Aug(\Lambda_{\Cliff})$ is defined by the augmentation
  for the Harvey-Lawson filling $L_{(1)}$ for some choice of spin structure.
\end{theorem} 

Here the Harvey-Lawson filling is an example of an asymptotically-cylindrical Lagrangian filling of the Clifford Legendrian.  Let
\[ a_1,\ldots, a_n \in \R \]
be 
real numbers with exactly two being zero.  Define as in Joyce
\cite[(37)]{joyce}
\begin{equation} \label{hl}
L_{(1)} = \Set{ (z_1,\ldots, z_n) \in \C^n | \begin{array}{l} |z_i|^2 - a_i^2
    = |z_j|^2 - a_j^2 , \forall i, j  \\
                                       z_1 \ldots z_n \in
                                       [0,\infty) \end{array} } .
\end{equation} 
The proof proceeds by showing that the image of the restriction map to the boundary in cohomology of the moduli space for any filling would be too large for the image to be a maximally isotropic subspace.  

\begin{remark}
    We frequently use the results in the first two papers in this series \cite{BCSW1,BCSW2} and cite them with prefixes I- and II- respectively.
\end{remark}

\section{The augmentation variety}

In this section,  we define an analog of the augmentation variety of 
Ng \cite{ng:framed}; see also  Aganagic-Ekholm-Ng-Vafa \cite{aganagic},
Diogo-Ekholm \cite{diogo2020augmentations} and Gao-Shen-Weng \cite{gaoshenweng}
for definitions in other contexts.  We make various computations of augmentation varieties, and in particular prove Theorem \ref{augpot} from the introduction.

First, we recall the definition of the Chekanov-Eliashberg algebra \cite{BCSW1,BCSW2} in circle-fibered contact manifolds. Let $(Y,\omega_Y)$ be a closed symplectic manifold such that $[\omega_Y] \in H^2(Y,\Z)$ is an integral cohomology class. There exists a principle $S^1$-bundle $p:Z\to Y$ whose first Chern class is $- [\omega_Y] \in  H^2(Y,\Z)$ together with a connection one-form form $\kappa \in \Omega^1(Z)$ such that:
\begin{itemize}
        \item $\alpha = - \kappa$ is a contact form on $Z$;
        \item the curvature form of $\kappa$ is $
        - \omega_Y$, that is, $\d\kappa = 
        - p^*\omega_Y$;
        \item and the generating vector field defining the principal $S^1$-action on $Z$ coincides with opposite of the Reeb vector field $R_\alpha$ of $\alpha$.
    \end{itemize}
We call such a contact manifold $(Z,\alpha)$ a \textit{circle-fibered contact manifold}. For a Legendrian submanifold $\Lambda$ in $Z$, we write $\Pi:=p(\Lambda)\subset Y$, which is a (possibly immersed) Lagrangian in $Y$.

The Reeb orbits of this contact form $\alpha$ are multiple covers of the circle fibers. Our convention identifies a circle fiber with $\R/\Z$. In other words, the action of a simple Reeb orbit is one.

A circle-fibered contact manifold $(Z,\alpha)$ with base $(Y, \omega_{Y})$ is called {\em  tame} if for some constant $\tau_Z \ge 1$ and integral symplectic form
\[ \omega_{Y,0} \in \Omega^2(Y,\R), \quad [\omega_{Y,0}] \in H^2(Y,\Z)  \] 
we have
\[   \omega_Y = \tau_Z \omega_{Y,0} \]
and the base is monotone in the sense that 
there exists a {\em  monotonicity constant} $\tau_Y \ge 3$ so that 
\[ c_1(Y) = \tau_Y [\omega_{Y,0}] \in H^2(Y,\R) .\] 
For any compact spin embedded Legendrian $\Lambda \subset Z$ with embedded
image $\Pi \subset Y$ we call $(Z,\Lambda)$ a {\em  tame pair}. The Chekanov-Eliashberg algebra for a tame pair is constructed in \cite{BCSW2}. Now we recall its definition. Note that the contact form $\alpha$ is of Morse-Bott type, the spaces of Reeb chords with ends on $\Lambda$ are disjoint union of smooth manifolds. The following provides a Morse model of the Chekanov-Eliashberg algebra.

A {\em Morse datum} for $(Z,\Lambda)$ consists of a pair  of vector fields on the space of Reeb chords $\cR(\Lambda)$ and on the cylinder on the Legendrian 
\[ \zeta_\white \in \Vect({\cR}(\Lambda)),  \quad  \zeta_\black \in \Vect(\R \times \Lambda)^{\R} \] 
arising as follows:
\begin{enumerate} 
\item  There exists a  Morse function on the space of Reeb chords
\[  f_\white:  {\cR}(\Lambda) \to \R ;\] 
so that $\zeta_\white$ is the gradient vector field:
\begin{equation} \label{zdiam} 
\zeta_\white := \grad(f_\white) \in \Vect({\cR}(\Lambda)) .\end{equation}
\item There exists a Morse function 
\[ f_\black :   \Lambda \to \R ;\] 
with gradient vector field 
\[ \grad(f_\black) \in \Vect(\Lambda) \] 
so that $\zeta_\black$  is a translation-invariant lift of $\grad(f_\black)$\label{zwhite}.
\end{enumerate}

A vector field $\zeta_\black \in \Vect(\R \times \Lambda)$ is {\em positive} if there exists a function
\[ a: \Lambda \to \R_{> 0} \]
so that 
\begin{equation} \label{zetablack} \zeta_\black = a(\lambda) \partial_s + {p}^* \grad(f_\black) 
\end{equation}
where 
\[ {p}^* : \Vect(\Lambda) \to \Vect(\R \times \Lambda)^\R  \] 
is the obvious identification of
translationally-invariant vector fields trivial in the $\R$-direction
with 
vector fields on $\Lambda$.

The limit of any Morse trajectory along any infinite length trajectory is a zero of the gradient vector field.   We introduce labels for the possible limits of the trajectories above as follows. 
Denote by 
\[ \ul{\R} \cong T\R \times \Lambda \] 
the translational factor in
$T(\R \times \Lambda) = T\R \oplus T \Lambda$.  The  zeroes of the vector field
$p_*(\zeta_\black)$ correspond to tangencies of $\zeta_\black$ with the translational factor:
\[  {p}_*(\zeta_\black)^{-1}(0) = \zeta_\black^{-1}(\ul{\R}) \subset
  \R \times \Lambda . \] 

\begin{definition} The union of the zeroes of the vector fields is denoted 
\begin{equation} \label{gens2} \cI(\Lambda) := 
\cI_\white(\Lambda) \cup \cI_\black(\Lambda), 
\quad  \cI_\white(\Lambda) := \zeta_\white^{-1}(0), \quad \cI_\black(\Lambda) := \zeta_\black^{-1}(\ul{\R}).
\end{equation}
Let $\cW(\Lambda)$ denote the space of ordered words in the generators
$\cI(\Lambda)$ defined as above:
\begin{equation} \label{words} 
\cW(\Lambda) = \bigcup_{d \ge 0} \cI(\Lambda)^d . \end{equation}
For any word $w  = \gamma_1 \ldots \gamma_k \in \cW(\Lambda)$ denote by 
\[ \ell(w) = k \in \Z_{\ge 0} \] 
the length of $w$ and by $\ell_\black(w)$ the number of {\em classical generators}
 $\gamma_i \in \cI_\black(\Lambda)$. The space of {\em contact chains} is 
the completion 
\begin{equation} \label{CE} CE(\Lambda) = \Set{ \sum_{i=1}^\infty c_i
    \Sigma_i |  \ \Sigma_i \in \cW(\Lambda), c_i \in \mhat{G}(\Lambda), \quad \lim_{i \to \infty}( \ell_{\black}(\Sigma_i)) = \infty }
\end{equation} 
of $\mhat{G}(\Lambda)$-valued functions on $\cW(\Lambda)$ with
respect to the filtration given by the classical word length $\ell_\black$. Here $\mhat{G}(\Lambda)$ is the completion of the group ring $\C[H_1(\Lambda)]$ over the first homology $H_1(\Lambda)$. If $\Pi=p(\Lambda)$ is a monotone Lagrangian, we can just use the the group ring $G(\Lambda)=\C[H_1(\Lambda)]$.

We also call $CE(\Lambda)$ the Chekanov-Eliashberg algebra of the Legendrian $\Lambda$, with multiplication given by concatenation of words; the
empty word is the unit $1$. Denote by
\[ CE_{\ell } (\Lambda) =  \bigoplus_{ \ell(w) = \ell}
\mhat{G}(\Lambda) w \subset CE(\Lambda)   \] 
the subspace generated by words $w$ of length $\ell$.    This ends the Definition.
\end{definition} 


\begin{definition} {\rm (Gradings)} Suppose that $\Pi \subset Y$ is monotone with monotonicity constant $\tau \in \R$.  Define the {\em real grading}
\[ \cI_\white (\Lambda) \to \R, \quad \gamma \mapsto \deg_\R(\gamma) = \on{ind}
(f_\white)(\gamma) + \tau \theta - 1 \] 
where $\theta$ is the action of the Reeb chord $\gamma$.  Define
\[ \cI_{\black}(\Lambda) \to \R, \quad \gamma \mapsto \deg_\R(\gamma) :=
\on{ind} (f_\black)(\gamma)- 1  . \]
For words define  the degree map as the sum of the degrees
of the factors:
\[ \cW(\Lambda) \to \R, \quad \gamma_1 \otimes \ldots \otimes \gamma_k \mapsto
  \deg_\R(\gamma_1 \otimes \ldots \otimes  \gamma_k) := \sum_{i=1}^k \deg_\R(\gamma_i) .\]
If $\Lambda$ is connected, then the actions of Reeb chords are multiples of $2/\tau$ and the real 
grading defines an $\Z$-grading.
%
We also recall that the generators have a $\Z_2$ grading $\deg_{\Z_2}$ which records the parity of the corresponding Morse index.
\end{definition}

\begin{definition}  An {\em augmentation}  is a chain algebra map
\[ \varphi: CE(\Lambda) \to G(\varphi) \] 
for some commutative $G(\Lambda)$-ring $G(\varphi)$, considered as a trivial complex. A {\em graded augmentation} is defined similarly, by requiring that $\varphi$ is a 
dga map and $G(\varphi)$ is concentrated in degree zero.  
\end{definition}

In particular, any augmentation must vanish on the
image of the differential $\delta$.    By the second part \cite{BCSW2} in this series, any tamed Lagrangian filling $L$ of $\Lambda$ gives rise to  an augmentation with target $G(\varphi) = \hat{G}(L)$
where $\hat{G}(L)$ is a completed group ring on the first homology $H_1(L)$.

\begin{example}  \label{algaug} Consider the dga $CE(\Lambda)$ for the Clifford
Legendrian $\Lambda \subset S^{2n-1}$ where $\Lambda$ has the standard spin
structure and is equipped with the standard Morse data.  The generators 
\[ \cc_1,\ldots,\cc_{n-1} \in \cI_\black(\Lambda) \] 
from II-Example \ref{II-cliffleg6} corresponding to the critical points of index one are degree zero.  These generators may be mapped to non-zero elements in $G(\varphi)$ under any graded augmentation.  The degree one generators are the 
classical generator $\bb \in \cI_\black(\Lambda)$ of Morse degree two, 
and the Reeb chord $\aa \in \cI_\black(\Lambda)$ of Morse degree zero 
and length $2\pi/n $.  The differentials of these degree one generators
are given by 
\begin{equation} \label{aarel}
\delta^{\ab,0}(\bb) = 0 , \quad \delta^{\ab,0}(\aa) = 
\pm 1 \pm  [\mu_1] \exp(\cc_1) \pm \ldots  \pm  [\mu_{n-1}] \exp(\cc_{n-1}) \end{equation}
where the signs depend on the choice of spin structure.
Define as coefficient ring
\[ G(\varphi) = \C [[ [\mu_1],\ldots,[\mu_{n-1}]]] .\]
Define a map $\varphi: CE(\Lambda) \to G(\varphi)$ by 
\[ \varphi(\mu_n) = \pm 1 \]
where the sign is chosen so that the relation 
given by \eqref{aarel} becomes 
\[ \varphi(\exp(\cc_{n-1})) = \varphi(1 \pm [\mu_1] \pm \ldots \pm [\mu_{n-2}]) ; \] 
this guarantees
that the logarithm 
\[ \varphi(\cc_{n-1}) := \ln(1 \pm [\mu_1] \pm \ldots \pm [\mu_{n-2}]) \] 
has zero constant term and is so well-defined in the completed coefficient ring.
Then $\varphi$ defines a graded augmentation.   This ends the Example.
\end{example}

We wish to extract from the space of augmentations
a subvariety of the abelian representation variety. 

\begin{definition} \label{augvar} {\rm (Augmentation variety)}
Let $(Z,\Lambda)$ be a tame pair. The {\em extended augmentation ideal}
\[ \ti{I}(\Lambda) = \bigcap_\varphi  {(\iota\circ \varphi)}^{-1}(0_{G(\varphi)})  \subset CE(\Lambda) \] 
is the set of elements
in $CE(\Lambda)$ in the kernel $\ker(\iota \circ \varphi)$ of every  augmentation 
\[ \varphi: CE(\Lambda)\otimes G(\varphi) \to G(\varphi) \]
where the coefficient ring $G(\varphi)$ is some abelian ring and \[\iota: CE(\Lambda) \to CE(\Lambda)\otimes G(\varphi), \;  a \mapsto a\otimes 1 \] is the tensoring-by-one morphism.
%
%
Denote {\em the abelianized complex}
\[ CE^{\ab}(\Lambda) = CE(\Lambda)/\sim \] 
to be the quotient of $CE(\Lambda)$ obtained by identifying two words that are equivalent up to re-ordering up to a sign determined by the grading of the letter
\[ a b \sim (-1)^{|a||b|}ba . \]
Denote by 
\[ \delta^{\ab}: CE^{\ab}(\Lambda)  \to CE^{\ab}(\Lambda)  \]
the map induced by the differential $\delta$.  Since each ring $G(\varphi)$ is abelian, 
$\widetilde{I}(\Lambda)$ is invariant under permutation of any elements in the constituent words, 
and so is the inverse image of an abelianized ideal
\[ \widetilde{I}^{\ab}(\Lambda) = \widetilde{I}(\Lambda) /\sim . \]
Suppose that $\Lambda$ is monotone, and $CE(\Lambda)$ is defined using the group ring
$\hat{G}(\Lambda)$ on $H_1(\Lambda)$. Choose 
\[ \cc_1,\ldots,\cc_k \in CE(\Lambda) \] 
in the span of the 
generators $\cI_\black(\Lambda)$ representing a basis of codimension one cycles 
given by $\mu_1,\ldots,\mu_k \in H_1(\Lambda)$.
Embed  the ring of functions on $\Rep(\Lambda)$ in $CE^{\ab}(\Lambda)$ by assigning to each weight 
\[ \lambda = (\lambda_1,\ldots, \lambda_k) \] 
the monomial 
\[ y^\lambda = y_1^{\lambda_1} \ldots y_k^{\lambda_k} \]
where 
\[ y_i  = [\mu_i]  \exp( \cc_i ) . \]
The {\em augmentation ideal} is 
  \[ I(\Lambda) = \pi(\widetilde{I}(\Lambda)) \cap \C[\Rep(\Lambda)]  .\] 
Let 
\[ \Aug(\Lambda) \subset \Rep(\Lambda)  \] 
be the variety defined by $I(\Lambda)$.   

\vskip .1in \noindent The $\R$-graded augmentation 
variety $\Aug_{\R}(\Lambda) \subset \Rep(\Lambda)$ 
is defined similarly, by allowing 
only $\R$-graded augmentations.  This ends the Definition.
\end{definition} 

\vskip .1in
\noindent One also might define the $\Aug(\Lambda)$ as the {\em scheme} defined by $I(\Lambda)$, rather than variety.   The arguments below will show that the scheme is also a Legendrian isotopy invariant,  but we avoid schemes since we have no application for this more refined invariant at the moment.

\begin{example} Let $\Lambda \subset S^{5}$ denote the Clifford Legendrian  with trivial 
 spin structure as in Example II-\ref{II-cliffleg6}.  
 Let $\aa \in CE(\Lambda)$ the degree one generator given by the minimal length chord, as in Example \ref{algaug}.  The differential
 \[ \delta^{\ab,0}(\aa)=  1 + [\mu_1] \exp(\cc_1) + [\mu_2] \exp(\cc_2) \] 
must map to zero under
 any augmentation $\varphi$, and so 
\begin{equation} \label{auginc}
\Aug(\Lambda) \subset \{ 1 + y_1 + y_2 \} . \end{equation} 
On the other hand, by Example \ref{algaug} there exists an augmentation
 which does not vanish on any polynomial in $y_1$, so 
 $\Aug_{L_{(1)}}(\Lambda)$ is a hypersurface.  It follows that the inclusion \eqref{auginc} is an equality.  \end{example} 

\begin{remark}  We have a natural inclusion
%
%
\[ \Aug_{\R}(\Lambda)  \subset \Aug(\Lambda) \]
induced by the reverse inclusion of ideals.  
\end{remark}

\begin{remark}  Multiple definitions of the augmentation variety appear in the literature. 
Ng \cite{ng:framed}, working with Legendrian two-tori, defines the augmentation variety as the space of points in the representation variety for which there exists an augmentation corresponding to that specialization of variables.   Gao-Shen-Weng \cite{gaoshenweng} define the augmentation variety as the  moduli space of augmentations.  Aganagic-Ekholm-Ng-Vafa \cite{aganagic} and  Diogo-Ekholm \cite{diogo2020augmentations} define the augmentation variety  as a subvariety of the quantized torus in the case of contact knot homology. 

A simpler version of the augmentation variety in our case would be  to consider the complex 
$CE_\white(\Lambda)$ generated by Reeb chords. 

The homology $HE_{\white}(\Lambda)$ has 
degree zero part $HE_{\white,0}(\Lambda)$ (with apologies for the notation) which is a commutative graded ring.
 We can therefore define the {\em lch spectrum} of prime ideals  
in $HE^{\ab}_{\white,0}(\Lambda)$ 
\begin{equation} \label{spectrum} 
\Aug_\white(\Lambda) = \on{Spec}(HE^{\ab}_{\white,0}(\Lambda)) . \end{equation}
This variety will also be shown to be a Legendrian isotopy invariant, but Lagrangian fillings do not necessarily define components of this augmentation variety.   There is a natural map 
\[ \Aug_\white(\Lambda) \to \Rep(\Lambda) \]
obtained from the embedding of functions in 
$HE_{\white,0}(\Lambda)$ as the coefficients of the zero length word. 
\end{remark}

\begin{proposition}\label{prop:cycle_independence}  Let $(Z,\Lambda)$ be a tame pair. 
    The augmentation varieties $\Aug(\Lambda)$ resp. $\Aug_\R(\Lambda)$ are independent of the choice of
    embedding $\C[\Rep(\Lambda)] \to CE^{\ab}(\Lambda,\hat{G}(\Lambda))$,
    that is, independent of the choice of cycles $\cc_1,\ldots,\cc_k $.
    \end{proposition}

\begin{proof}   We will check that any two choices have difference that is a coboundary of the Chekanov-Eliashberg differential, up to terms that vanish under abelianization or projection to degree zero words. 
    Let
    \[ \cc_1',\ldots, \cc_k' \in \cI_\black(\Lambda) \] 
    be another choice of cycle
    representatives for $H_1(\Lambda)_{\on{free}}$ so that $[\cc_i] = [\cc_i']$ for all $i$. Denote the degree 1 monomials of the embedding $\C[\Rep(\Lambda)] \to CE^{\ab}(\Lambda,\hat{G}(\Lambda))$ using $\cc_i$ as $w_i$. Similarly, denote the monomials of the embedding using $\cc_i'$ as $w_i'.$
    By definition, there exist chains
    \[ \bb_i \in CE_\black(\Lambda ) \]
    of Morse degree two, and so real degree one, with
    \[ \delta_{\on{Morse}}(\bb_i) = \cc_i - \cc_i' .\]
    The contact differential applied to $\bb_i$ gives these terms plus additional terms involving constant holomorphic disks:
    \[ \delta(\bb_i) = \cc_i - \cc_i' + \sum_{j_1,\ldots, j_k} 
    \delta_d(\bb_i, \cc_{j_1},\ldots, \cc_{j_k}) \cc_{j_1} \ldots \cc_{j_k}
    \] 
    where $\delta_d(\bb_i, \cc_{j_1},\ldots, \cc_{j_k})$ is the coefficient
    of  $ \cc_{j_1} \ldots  \cc_{j_k}$ in $\delta_d(\bb_i)$, with notation from Proposition II-\ref{II-abelian}. By Proposition \ref{II-abelian}, we may assume that the structure coefficients 
\[ \delta_d(\bb_i, \cc_{j_1},\ldots, \cc_{j_k})  \in \C \subset \hat{G}(\Lambda) \]
of the Morse \ainfty algebra
of $\Lambda$ make the degree zero terms skew-symmetric. The skew-symmetric words  lie in the kernel of $\varphi$ since $\varphi$ is an algebra homomorphism and $\hat{G}(L)$ is abelian.  On the other hand,  $\varphi$ is graded and so vanishes on negative degree generators.  
Thus

\[ 0 = \varphi(\delta(\bb_i)) = \varphi(\cc_i) - \varphi(\cc_i'). \]
Moreover, since $\varphi$ is a dga map, 
$$ \varphi(p(y_1,\dots, y_k)) = \varphi(p(y'_1,\dots, y'_k)) $$
for any polynomial $p$ where $e^{w_i}=y_i$ and $e^{w'_i}=y'_i$. Finally, from the above equation, it readily follows that the augmentation ideal $I(\Lambda)$ is independent of the choice of cycles $\cc_1,\dots,\cc_k$.
\end{proof}

\begin{lemma}\label{singtoall}
    Let $(Z,\Lambda)$ be a tame pair and let $\varphi : CE(\Lambda,G(\varphi)) \to G(\varphi)$
    resp. $\varphi^{\ab} : CE^{\ab}(\Lambda,G(\varphi)) \to G(\varphi)$
        be an algebra map such that $\varphi(\delta(\aa)) = 0$ for all single letter generators $\aa \in \cI(\Lambda)$.  Then $\varphi$ resp. $\varphi^{\ab}$ is an augmentation.  
\end{lemma}

\begin{proof}
    We need to check $\varphi \circ \delta = 0$.  By the Leibniz rule,
    \[ \varphi(\delta(\aa_1 \aa_2)) = \varphi(\delta (\aa_1) \aa_2 \pm \aa_1 \delta(\aa_2)) =  \varphi(\delta (\aa_1) ) \varphi(\aa_2) \pm \varphi(\aa_1 ) \varphi(\delta(\aa_2))) =0  .\]
The lemma for $\varphi$ follows from an inductive argument on word length; the proof for $\varphi^{\ab}$ is similar. 
\end{proof}

\section{Relation to the disk potential}

In preparation for the construction of orientations we make the following splittings. 
Recall that $\cc_1,\ldots, \cc_k$ is a collection of cycles in $C_1(\Lambda)$ representing a basis in homology $H_1(\Lambda)$.   We write 
\[ \delta_{M,k} : C_k(\Lambda) \to C_{k-1}(\Lambda), \quad k \ge 0 \]
for the Morse differential, as usual taken with complex coefficients.  
We view $C_1(\Lambda)$ as a subspace of $CE_0(\Lambda)$, via the embedding as words
of length one.    Since $C_1(\Lambda)$ is finite dimension, there exist  splittings
\begin{equation} \label{eq:split} C_1(\Lambda) \cong \on{ker}(\delta_{M,1}) \oplus \on{im}(\delta_{M,1}) \cong \on{im}(\delta_{M,2}) \oplus H_1(\Lambda) \oplus  \on{im}(\delta_{M,1}). 
\end{equation}
This extends in an obvious way to a splitting with $G(\Lambda)$-coefficients.  
Given a ring $G(\varphi)$ and a map from (singular) homology of $\Lambda$,
\[ \varphi: H_1(\Lambda,G(\Lambda)) \to G(\varphi) \] 
we may extend to $C_1(\Lambda,G(\varphi))$ by defining $\varphi$ to be zero 
on the first and third factors in the last splitting in \eqref{eq:split}.
The map $\varphi$ then extends uniquely to words of all length in $CE_0(\Lambda)$ by the 
algebra homomorphism property. 

\begin{lemma} \label{lem:vanish} Let $(Z,\Lambda)$ be a tame pair with $\Lambda$ connected. Suppose that a map $\varphi: CE_0^{\ab}(\Lambda) \to G(\varphi)$ defined as above satisfies the property that $\varphi$ vanishes on the image of the restriction of $\delta$ to the degree one Reeb chords $\aa \in CE_1(\Lambda)$ of minimal action.  Then $\varphi$ defines an augmentation.
\end{lemma}

\begin{proof} 
Clearly, one just needs to verify that the map $\varphi$ vanishes on the image of the contact differential, $\delta$, to prove that $\varphi$ is an augmentation. 

Using linearity and degree considerations, it is enough to verify $\varphi $ vanishes on the image, $\delta(w)$, of monomial words $w$ of degree 1. It will be enough to consider words $w$ that have a degree 0 generator. Indeed, if the word $w$ does not have any degree 1 generator, then $\delta(w)$ will be a linear combination of words that do not have a degree zero generator. From the algebra homomorphism property of $\varphi$ we see that $\varphi(\delta(w))$ vanishes for such words.

We deal with the case of words $w$ which contain degree 1 generators. Such words have an expression $\mathfrak{l}\aa$ where $\aa$ is a degree one Reeb chord generator. By using the Leibniz property, we have
\[\delta(\mathfrak{l}\aa) = \delta(\mathfrak{l})\aa \pm \mathfrak{l}\delta(\aa).\]
%
%
Since $\varphi$ vanishes on the Reeb chord generators, the first term in the above sum vanishes. The second term vanishes from the hypothesis of the Lemma. Thus $\varphi$ is a chain map.
\end{proof}

 \begin{figure}[ht]
     \centering
     {\tiny 
     \scalebox{1}{
\begingroup%
  \makeatletter%
  \providecommand\color[2][]{%
    \errmessage{(Inkscape) Color is used for the text in Inkscape, but the package 'color.sty' is not loaded}%
    \renewcommand\color[2][]{}%
  }%
  \providecommand\transparent[1]{%
    \errmessage{(Inkscape) Transparency is used (non-zero) for the text in Inkscape, but the package 'transparent.sty' is not loaded}%
    \renewcommand\transparent[1]{}%
  }%
  \providecommand\rotatebox[2]{#2}%
  \newcommand*\fsize{\dimexpr\f@size pt\relax}%
  \newcommand*\lineheight[1]{\fontsize{\fsize}{#1\fsize}\selectfont}%
  \ifx\svgwidth\undefined%
    \setlength{\unitlength}{118.8874054bp}%
    \ifx\svgscale\undefined%
      \relax%
    \else%
      \setlength{\unitlength}{\unitlength * \real{\svgscale}}%
    \fi%
  \else%
    \setlength{\unitlength}{\svgwidth}%
  \fi%
  \global\let\svgwidth\undefined%
  \global\let\svgscale\undefined%
  \makeatother%
  \begin{picture}(1,0.70579393)%
    \lineheight{1}%
    \setlength\tabcolsep{0pt}%
    \put(0,0){\includegraphics[width=\unitlength,page=1]{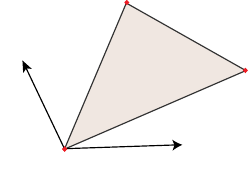}}%
    \put(-0.00381959,0.18109061){\makebox(0,0)[lt]{\lineheight{1.25}\smash{\begin{tabular}[t]{l}$e_1$\end{tabular}}}}%
    \put(0.45111455,0.00889596){\makebox(0,0)[lt]{\lineheight{1.25}\smash{\begin{tabular}[t]{l}$e_2$\end{tabular}}}}%
    \put(0.2183328,0.03653218){\makebox(0,0)[lt]{\lineheight{1.25}\smash{\begin{tabular}[t]{l}$0$\end{tabular}}}}%
  \end{picture}%
\endgroup%
}}
     \caption{Basis to ensure non-negative exponents}
     \label{fig:basisforlam}
 \end{figure}
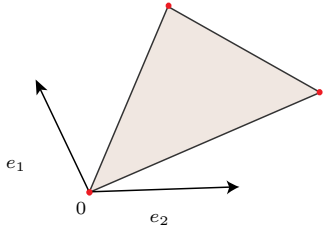

We now prove the relation to the zero level set of the disk potential.
Let $W_\Pi: \Rep(\Pi) \to \C$ be the disk potential of $\Pi$.
As in \cite[\eqref{II-Wlam}]{BCSW2}, 
$W_\Pi$ is, up to a shift by a monomial, 
the lift of a function $W_\Lambda: \Rep(\Lambda) \to \C$.

\begin{theorem} \label{algpotthm} 
Let $(Z,\Lambda)$ be a tame pair
with $\Lambda$ connected.  The  augmentation variety 
of $\Lambda$ resp. lch spectrum satisfies 
\[ \Aug(\Lambda) = \Aug_\white(\Lambda) =  W_\Lambda^{-1}( 0)  . \] 
\end{theorem}

\begin{proof}   One direction of containment  for the augmentation variety
follows from the partial computation of the differential 
in Example II-\ref{II-toricex3}:  If $\aa \in \cI_\white(\Lambda)$ 
is the Morse-degree-zero generator of minimal Reeb length then
\begin{equation} \label{augconstraint}
\varphi( \delta^{\ab}(\aa)) =   W_\Lambda(\varphi(^{\ab}[\mu_1])e^{\varphi(\cc_1)},\dots, \varphi([\mu_{k}])e^{\varphi(\cc_{k})}) .\end{equation}
Thus 
\[ \Aug(\Lambda) \subseteq W_\Lambda^{-1}( 0)  . \] 

To show the reverse inclusion, we explicitly construct an 
augmentation valued in a certain formal power series ring
which takes the desired values given by any point in the augmentation variety.  We may without loss of generality assume that 
a basis for $H_1(\Lambda)$ has been chosen so that the exponents in \eqref{augconstraint} are non-negative.

We find an augmentation by a formal expansion around a transversally-cut-out zero of the disk potential.  Let 
\[ W_\Lambda(y_{k}) = W_\Lambda(0,\ldots,0,y_{k}) \]
denote the polynomial obtained by setting the first $k-1$ coordinates to zero.
Then $ W_\Lambda(y_{k}) $ is a polynomial in $y_{k}$, and as such has 
at least one solution over the complex numbers, call it $\kappa \in \C$, so that 
\[ W_\Lambda(\kappa) = 0 . \]
As a simple case of Sard's theorem, for a generic
linear transformation in $GL(k,\C)$ in the coordinates $y_1,\ldots,y_{k}$, each non-zero solution $\kappa$ is transversally cut out.  
Equivalently, that is, the roots of the polynomial 
$W_\Lambda(\kappa)$ 
have multiplicity one.  Indeed, this condition is equivalent to the condition that a generic line intersects $W^{-1}_\Lambda(0)$ transversally.
The projection 
\[ \pi: W_\Lambda^{-1}(0) - \{ 0 \} \to \CP^{k-1} \]
onto complex projective space 
$\CP^{k-1}$ has finite fiber over any line $\ell \in \CP^{k-1}$.  Indeed, if not the fiber $\pi^{-1}(\ell)$ would be a line and so contain $0$, which violates the condition that the constant term in 
$W$ is non-vanishing. So the image  $\pi(W_\Lambda^{-1}(0) - \{ 0 \}) $ is a quasiprojective 
variety of dimension $k-1$ and so dense.   The fiber over 
a generic  line in $\CP^{k-1}$ has the desired property, by Sard's theorem. 
Let 
\[ G(\varphi) = \C[[[\mu_1],\ldots,[\mu_{k-1}]]] .\] 
Define a map $\varphi$ from $CE_0(\Lambda)$ 
to $G(\varphi)$ by setting 
\[ \varphi([\mu_{n-1}]) = \kappa \]
\[ \varphi(\cc_1) = \ldots = \varphi(\cc_{k-1}) = 0 \]
and 
\[ \varphi(\cc_{k})  \in G(\varphi) \] 
so that 
\begin{equation} \label{sothat}
W_\Lambda([\mu_1],\ldots,[\mu_{k-1}],\kappa \exp(\varphi(\cc_{k}))) = 0 .\end{equation}
The existence of such a formal solution  follows from a formal version of the implicit function theorem, that is, an order-by-order analysis:  The leading order term vanishes by assumption.
Consider the linearization 
\[ D_{y_k} W_\Lambda(\kappa): \C[\mu_1,\ldots,\mu_{k-1}]
\to \C[\mu_1,\ldots,\mu_{k-1}]  \] 
given by multiplication by the number $D_{y_k} W_\Lambda(\kappa) \in \C^\times$ of \eqref{sothat}.  
This map is an isomorphism, as the solution 
$(0,\ldots, 0,\kappa)$ is transversally cut out.
\footnote{Here we use the fact that $CE(\Lambda)$ is defined over $G(\Lambda)$
rather than  the completion $\hat{G}(\Lambda)$; if we used $\hat{G}(\Lambda)$
then we would have to justify that $\varphi$ is well-defined on the completion 
which is unclear to us.   }
Given a solution $\ss_{d}$ of \eqref{sothat} to order $d$ define 
\[ \ss_{d+1} =  \ss_{d} + ( D_\kappa W_\Lambda)^{-1}( 
W_\Lambda([\mu_1],\ldots,[\mu_{k-1}],\kappa \exp(\ss_{d})) \]
Then $\ss_{d+1}$ is a solution to order $d+1$ and agrees
with $\ss_{d}$ to order $d$. Taking the limit gives the desired solution
$\ss = \varphi(\cc_k)$. 
By construction, $\varphi$ is non-vanishing on the ring generated by the
first $k-1$ coordinates $y_1,\ldots, y_{k-1}$, so the subvariety
$\Aug_\varphi(\Lambda)$ defined by the ideal $\varphi^{-1}(0)$ is a hypersurface
containing $(0,\ldots, 0, \kappa)$.   By construction $\varphi$ vanishes on the differentials of the minimal length Reeb chords.  Since $\delta$ on the abelianization agrees with the Morse differential, by Lemma II-\ref{II-lem:noconst}, it also vanishes on images of the differential
applied to the classical generators.   By Lemma \ref{lem:vanish}, $\varphi$ is an augmentation.   Thus the irreducible component of  $W_\Lambda^{-1}( 0)  $ containing  $(0,\ldots, 0, \kappa)$ is contained in $\Aug(\Lambda)$. 

Repeating this procedure for each irreducible component of $\Rep(p) (W^{-1}_\Pi( 0) )$ proves that all irreducible components are so contained.  More precisely, 
suppose that 
\[ W_\Lambda = W_{\Lambda,1}^{d_1} W_{\Lambda,2}^{d_2} \ldots W_{\Lambda,l}^{d_l} \]
is the decomposition of $W_\Lambda$ into irreducible factors with multiplicities
$d_1,\ldots, d_l$.  For each $i = 1,\ldots, l$ choose coordinates and $\kappa \in \C$ so that $(0,\ldots, 0,\kappa)$ is a transversally cut out solution 
to $W_{\Lambda,i} = 0$ not contained in any other $W_{\Lambda,j}^{-1}(0)$.    The construction of the previous paragraph gives an augmentation such that $\Aug_{\varphi_i}(\Lambda) \subset \Aug(\Lambda)$ is a hypersurface containing
$W_{\Lambda,i}^{-1}(0)$.  Thus $\Aug(\Lambda)$ contains each of the irreducible
components of $\Rep(p)(W^{-1}_\Pi(0))$, and this proves the equality claimed in the Theorem.

It remains to show that lch spectrum is also given by the zero level set of the
disk potential.  By degree considerations,
the elements $\delta_\white(\aa)$ generate
the image of $\delta_\white$ because $\aa$ is the only Reeb generator of degree one.  It follows that $\Aug_\white(\Lambda)$ is the variety defined by 
 $\delta_\white(\aa)$.  
\end{proof}

\begin{remark}
    
    We can modify the Definition \ref{augvar} by considering the augmentations to $G(\varphi)$ where $G(\varphi)$ is an integral domain. We call this ideal the \textit{extended domain-augmentation ideal} and denote it with $\widetilde I_D$. We can similarly define the domain-augmentation ideal, $I_D$, by projecting to the image of $\C[\Rep(\Lambda)]$. The proof of Theorem \ref{algpotthm} shows that under the same hypothesis as the Theorem, 
    \[I_D = \sqrt{\langle W_\Lambda  \rangle} .\]
\end{remark}

    One can also show that under the same hypothesis, Theorem \ref{algpotthm} can be extended to show equality of the augmentation schemes instead of varieties.

\begin{theorem} \label{augidealthm} 
Let $(Z,\Lambda)$ be a tame pair
with $\Lambda$ connected.  The  augmentation ideal 
of $\Lambda$ satisfies 
\[ I  = \langle W_\Lambda \rangle .\]
\end{theorem}
    \begin{proof}
        
    We give a quick sketch of the proof. Assume that  the following  factorization holds as before:
    \[W_\Lambda = W_{\Lambda,1}^{d_1} W_{\Lambda,2}^{d_2} \ldots W_{\Lambda,l}^{d_l}.\] 
    We proceed as in Theorem \ref{algpotthm}, but instead of constructing an augmentation $\varphi$ which vanishes on one of the irreducible factor $W_{\Lambda,i}$ of $W_\Lambda$, we construct an augmentation $\varphi$ such that $\varphi(W_{\Lambda,i})$ is a nilpotent element of order $d_i$. Let 
\[ G(\varphi) = R[[\mu_1],\dots,[\mu_{k-1}] ], \;\text{ where} \quad R = \C [\alpha]/\langle{\alpha^{d_i}}\rangle. \] 
After a change in basis as in the proof of Theorem \ref{algpotthm},
we may assume that $W_{\Lambda,i}(0,..,y_k)$ is a polynomial with only transverse roots and let $1$ be a root. Thus $W_{\Lambda,i}'(0,..,1) \neq 0.$ Then, by viewing $W_{\Lambda,i}$ as a polynomial with $R$ coefficients, we have 
    \[W_{\Lambda,i}(0,..,1+\alpha) = c\alpha + \text{higher order terms in }\alpha, \;\; c \in \C^{\star} .\]
    Note that $W_{\Lambda,i}(0,..,1+\alpha)$ is a nilpotent element of order $d_i$. By performing an implicit function theorem inspired induction as before, we can create an augmentation 
    \[\varphi: CE(\Lambda) \to R[[[\mu_1],\dots, [\mu_{k-1}]]]\] 
    such that 
    \begin{align*}
    \varphi(\cc_{i})&=0 \; \forall i<k \\
    \varphi([\mu_i]) &= [\mu_i] \; \forall i<k \\
    \varphi ([\mu_k]) &= 1 \\
    \end{align*}
    and $\varphi(\cc_k)$ satisfies the equation 
    \[ W_{\Lambda,i}([\mu_1],\dots, [\mu_{k-1}],\exp(\varphi(\cc_k)) = W_{\Lambda,i}(0,..,1+\alpha). \]
    Thus we have that the ideal, $I_{\varphi}$, corresponding to the augmentation $\varphi$ contains $W_{\Lambda,i}^{d_i}$ but not any lower order exponents. Thus, by repeating this argument for each irreducible factor, we can conclude that    
    \[I = \langle W_\Lambda \rangle.\]

    \end{proof}

\begin{example}  We construct an augmentation for the Clifford Legendrian
with the trivial (unfillable) spin structure.   For the minimal length Reeb chord $\aa$ we have 
\[ \delta^{\ab}(\aa) =
1 + [\mu_1] e^{\cc_1} + [\mu_2] e^{\cc_2} . \]
We choose as solution the element 
\begin{equation} \label{given}
    \varphi([\mu_1]) = 0, \quad \varphi([\mu_2]) = \kappa = -1 . \end{equation}
A formal solution is then given by 
\[  \varphi(\cc_2) = \ln(1 + [\mu_1])  . \] 
On the other hand, we could expand around the solution
\eqref{given}
\[ \varphi(\cc_1) = \ln(1 + [\mu_2]) . \]
If the spin structure was fillable, we obtain similar augmentations
with different sign choices; these different augmentations correspond to the different choices  of smoothing of the Harvey-Lawson cone as we will 
explain in the next section. 
\end{example}

We wish to show that the augmentation variety is an invariant
of Legendrian isotopy.  For this, we will show that the cobordism maps induce maps on augmentation varieties.  We first achieve a partial skew-symmetry for the cobordism maps.   Write 
\[ \varphi(\gamma) = \sum_\Gamma \varphi_{\Gamma}(\gamma) \] 
where 
$\varphi_\Gamma(\gamma)$ is the contribution from maps with domain type $\Gamma$.

\begin{definition}  The map $\varphi$ will be called {\em classically commutative} if  the map $\varphi = \varphi(L,b)$ applied to a classical generator $\bb \in \cI_\black(\Lambda)$
 has the following skew-symmetry property:  Suppose that $e_i,e_{i+1} \in \Edge_{\rightarrow}(\Gamma)$ are leaves of a tree $\Gamma$ incident on the same vertex
$v \in \Ver(\Gamma)$.  Then the transposition $\sigma_{i(i+1)}$ satisfies
\[ \sigma_{i(i+1)} \varphi_\Gamma(\bb)  = (-1)^{ \deg_{\Z_2}(\gamma_i) \deg_{\Z_2} (\gamma_{i+1})} 
\varphi_\Gamma(\bb) .\]
That is, the output of $\varphi$ is graded-commutative with respect to the transposition of the edges $e_i,e_{i+1}$. In particular, after abelianization the output of $\varphi$ on classical generators consists of length one words,  and $\varphi$ is the classical Morse continuation map.
\end{definition}

\begin{lemma}  \label{abelian3} For invariant perturbations in the sense of Remark II-\ref{II-inv}, the cobordism map $\varphi(L,b)$
is classically commutative.  
\end{lemma}

The proof is similar to that of Proposition II-\ref{II-abelian}.
For any map of dga's
\[ \varphi: CE(\Lambda_+,G(\varphi)) \to CE(\Lambda_-, G(\varphi)) \]
denote the abelianization 
\[ \varphi^{\ab} :CE^{\ab}(\Lambda_+,G(\varphi)) \to CE^{\ab}(\Lambda_-, G(\varphi)) .\] 
%
\begin{lemma} \label{subring} Let $X = \R \times Z$ be a symplectization and 
$L$ is the cobordism constructed from an isotopy of Legendrians 
in Example I-\ref{I-isotopy}. 
The abelianized cobordism map 
\[ \varphi{(L,b)}^{\ab}: CE^{\ab}(\Lambda_+,G(\Lambda_+)) \to CE^{\ab}(\Lambda_-, G(\Lambda_-)) \] 
preserves the sub-rings
$\C[\Rep(\Lambda_+)] \cong \C[\Rep(\Lambda_o)]$ up to elements of the 
augmentation ideal $\ti{I}^{\ab}(\Lambda_-)$, that is, 
\[ \varphi(L,b)^{\ab} ( \C[\Rep(\Lambda_+)] ) \subset \C[\Rep(\Lambda_-)] + \ti{I}^{\ab}(\Lambda_-).\]
\end{lemma} 

\begin{proof}  First note that we have an identification of 
coefficient rings.  Indeed since $L$ is diffeomorphic to $ \R \times \Lambda$ we have  natural
identifications
\[ \hat{G}(L) \cong G(\Lambda_-) 
\cong G(\Lambda_+) . \] 
Next we check that the coordinate rings on the representative
variety are preserved.  Let $u: S \to X$ be a disk bounding $L$ 
with no punctures.  The number of intersections 
of the boundary $\partial u$ with the geometric cycle given by the union of stable manifolds  $\Sigma^s_i$ corresponding to the Morse cycle $\cc_i$ is the intersection number
of $\partial u$ and the closure of $\Sigma^s_i$.
The intersection number is topological and independent of the choice of representative
of $ [\cc_i]$.   The augmentation $\varphi(L,b)$ is defined on 
the generators $ \cI_\black(\Lambda)$ by counts of parametrized trajectories without disk components. Indeed, since the disks have no incoming strip-like ends, 
the positivity conditions imply that there are no non-constant disks, 
so that the treed disks in $X$ bounding $L$ have zero area.  

By Lemma \ref{abelian3}, any collection of cycles $\cc_{i,+}, i = 1,\ldots, k$
maps under $\varphi{(L,b)}$ to a collection of cycles $\cc_{i,-}, i = 1,\ldots, k$, up to terms that vanish under the projection to the abelianization and so vanish under every augmentation.
It follows that the subspace $\C[\Rep(\Lambda)_+]$ in $CE(\Lambda_+, \hat{G}(\Lambda))$ generated by polynomials $y^\lambda$
is mapped  to $\C[\Rep(\Lambda_-)]  $ in $CE(\Lambda_-, \hat{G}(\Lambda))$, up to terms in $\tilde{I}^{\ab}(\Lambda_-)$, as claimed.
\end{proof}

\begin{theorem}  \label{legisotopy}
Let $X = \R \times Z$ is a symplectization and 
$L_{01}$ is the cobordism constructed from an isotopy of Legendrians 
$\Lambda_0$ to $\Lambda_1$ in Example I-\ref{I-isotopy}. 
The cobordism map 
\[ \varphi{(L_{01},b_{01})}: CE(\Lambda_0,G(\Lambda_0)) \to CE(\Lambda_1, G(\Lambda_1)) \] 
induces  isomorphisms of augmentation varieties and lch spectra
\[ \Aug(L_{01}):  \  \Aug(\Lambda_0) \to \Aug(\Lambda_1), \quad 
\Aug_\white(L_{01}): \  \Aug_\white(\Lambda_0) \to \Aug_\white(\Lambda_1) .
\]
\end{theorem}

\begin{proof}  We must show that the cobordism map induces a map of augmentation ideals. 
Since the cobordism maps $\varphi{(L_{01},b_{01})}$ (in this case $b_{01} = 0$) are chain maps,  if $\varphi'$ is an augmentation of $CE(\Lambda_1,G(\Lambda_1))$
then the composition $\varphi' \circ \varphi{(L_{01},b_{01})}$ 
is an augmentation of $CE(\Lambda_0,G(\Lambda_0)$.
By Lemma \ref{subring}, the subring $\C[\Rep(\Lambda_0)]$ is preserved by the chain map $\varphi(L_{01},b_{01})$. 
Moreover, if $a\in I(\Lambda_0)$ and $\varphi'$ is any augmentation of $CE(\Lambda_1,G(\Lambda_1)),$ we have $\varphi' (\varphi{(L_{01},b_{01})}(a)) = \varphi' \circ \varphi{(L_{01},b_{01})}(a)= 0 $.  Thus the augmentation ideal $I(\Lambda_0)$ is also preserved by $\varphi(L_{01},b_{01})$:
\[ \varphi(L_{01},b_{01}) ( I(\Lambda_0)) \subseteq I(\Lambda_1) . \] 
So  the chain map $\varphi(L_{01},b_{01})$ induces a map of augmentation varieties from $\Aug(\Lambda_0)$ to $\Aug(\Lambda_1)$.  
The reverse isotopy induces a map $\varphi(L_{10},b_{10})$ with 
\[ \varphi(L_{10},b_{10}) I(\Lambda_1) \subseteq I(\Lambda_0) . \] 
On the other hand, the composition 
\[ \vartheta := \varphi(L_{10},b_{10}) \circ  \varphi(L_{01},b_{01}) \]
on $CE^{\ab}(\Lambda_0) $ is chain homotopic to the identity via some chain homotopy $h$.
Hence for $a \in I(\Lambda_0)$, utilizing the fact that $\C[\Rep(\Lambda_0)] \subset \ker \delta,$
\[ \vartheta( a ) = a + h \delta( a) + \delta h (a)  = a + \delta h (a).  \] 

Since the image of $\delta$ lies in $I(\Lambda_0)$, we have $\vartheta(a) = a $
for all $a \in I(\Lambda_0)$.   Together with the similar claim for the reverse 
order of composition, this shows  $\varphi(L_{01},b_{01}) ( I(\Lambda_0)) = I(\Lambda_1)$. 
\end{proof}

Chanda-Hirschi-Wang \cite{chw:tori} extend Vianna's construction of monotone tori to higher dimensional projective spaces. In particular, for every Markov triple $(a,b,c)$, they construct a monotone Lagrangian torus $\ol T_{abc}$ in $\P^n$ and show that their disk potentials have distinct Newton polytope. By taking the Bohr-Sommerfeld lifts of these monotone Lagrangian tori, we obtain embedded Legendrian tori in $S^{2n-1}$. Denote the Bohr-Sommerfeld lift of the $n-1$-dimensional lifted Vianna torus corresponding to the Markov triple $(a,b,c)$ as $\Lambda^{abc}_{n-1}$.

\begin{lemma}\label{irredofpot}
    The augmentation polynomial, $W_{\Lambda^{abc}_{n-1}}$, of the $n$-dimensional Legendrian torus $\Lambda^{abc}_{n-1}$ is an irreducible polynomial.
\end{lemma}

\begin{proof}
The augmentation polynomial is defined as the polynomial $x^{-v}W_{\ol T_{abc}} $ where $v$ is a vertex of the Newton polytope of $W_{\ol T_{abc}}$. Thus, the augmentation polynomial has the same Newton polytope as that of $W_{\ol T_{abc}}$, up to a translation and change of basis. When $n=3$, since the Newton polytope of $W_{\ol T_{abc}}$ is a triangle, from the Irreducibility Criterion of \cite{gaopoly}, we have that $W_{\Lambda^{abc}_{2}}$ is an irreducible polynomial. For $n>3$ we use induction to finish the proof.
\begin{figure}[ht]
     \centering
     {\tiny 
     \scalebox{1.6}{
\begingroup%
  \makeatletter%
  \providecommand\color[2][]{%
    \errmessage{(Inkscape) Color is used for the text in Inkscape, but the package 'color.sty' is not loaded}%
    \renewcommand\color[2][]{}%
  }%
  \providecommand\transparent[1]{%
    \errmessage{(Inkscape) Transparency is used (non-zero) for the text in Inkscape, but the package 'transparent.sty' is not loaded}%
    \renewcommand\transparent[1]{}%
  }%
  \providecommand\rotatebox[2]{#2}%
  \newcommand*\fsize{\dimexpr\f@size pt\relax}%
  \newcommand*\lineheight[1]{\fontsize{\fsize}{#1\fsize}\selectfont}%
  \ifx\svgwidth\undefined%
    \setlength{\unitlength}{225bp}%
    \ifx\svgscale\undefined%
      \relax%
    \else%
      \setlength{\unitlength}{\unitlength * \real{\svgscale}}%
    \fi%
  \else%
    \setlength{\unitlength}{\svgwidth}%
  \fi%
  \global\let\svgwidth\undefined%
  \global\let\svgscale\undefined%
  \makeatother%
  \begin{picture}(1,0.66666667)%
    \lineheight{1}%
    \setlength\tabcolsep{0pt}%
    \put(0,0){\includegraphics[width=\unitlength,page=1]{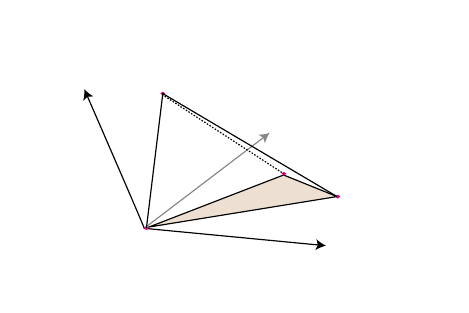}}%
    \put(0.17635496,0.28755966){\makebox(0,0)[lt]{\lineheight{1.25}\smash{\begin{tabular}[t]{l}$y$\end{tabular}}}}%
    \put(0.43695592,0.11906769){\makebox(0,0)[lt]{\lineheight{1.25}\smash{\begin{tabular}[t]{l}$x_1$\end{tabular}}}}%
    \put(0.35439484,0.27183378){\makebox(0,0)[lt]{\lineheight{1.25}\smash{\begin{tabular}[t]{l}$x_2$\end{tabular}}}}%
    \put(0.50042124,0.17916316){\makebox(0,0)[lt]{\lineheight{1.25}\smash{\begin{tabular}[t]{l}$\Newt(W_{\Lambda^{abc}_2})$\end{tabular}}}}%
    \put(0.40606572,0.4436956){\makebox(0,0)[lt]{\lineheight{1.25}\smash{\begin{tabular}[t]{l}$\Newt(W_{\Lambda^{abc}_3})$\end{tabular}}}}%
  \end{picture}%
\endgroup%
}}
     \caption{Newton polytope for two and three-dimensional tori corresponding to to $(a,b,c)$.}
     \label{fig:irredpoly}
 \end{figure}
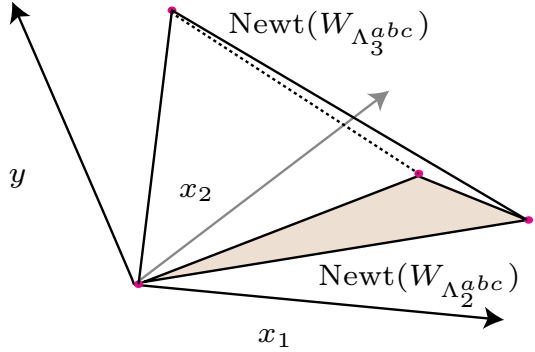
Assume that $W_{\Lambda^{abc}_{d-1}}$ is irreducible for $d>3$.  By Proposition 4.6 of Chanda-Hirschi-Wang \cite{chw:tori}, the set $\Newt(W_{\Lambda^{abc}_{d}})$ is a $d$ simplex obtained from a suspension of $\Newt(W_{\Lambda^{abc}_{d-1}})$. Pick a basis of $\Z^{d-1}$ such that the face $\Newt(W_{\Lambda^{abc}_{d-1}})$ lies in the cone generated by the first $d-2$ coordinates and $\Newt(W_{\Lambda^{abc}_{d}})$ lies in the positive cone corresponding to the basis. Call the polynomial variables corresponding to this choice of basis $(x_1,\dots,x_{d-1},y)$. By the choice of basis, setting $y=0$ in $W_{\Lambda^{abc}_{d}} (x_1,\dots,x_{d-1},y)$ recovers the augmentation polynomial $W_{\Lambda^{abc}_{d-1}}$. See Figure \ref{fig:irredpoly}.  If the polynomial $W_{\Lambda^{abc}_{d}} (x_1,\dots,x_{d-1},y)$ was reducible, then we will have \[W_{\Lambda^{abc}_{d}} (x_1,\dots,x_{d-1},y) = f(x_1,\dots,x_{d-1},y)g(x_1,\dots,x_{d-1},y).\]
Here $f,g$ are non-constant polynomials. Then, from \cite{ostro} (or Lemma 2.1 in \cite{gaopoly}) we have 
\[ \Newt(W_{\Lambda^{abc}_{d}}) = \Newt(f) + \Newt (g), \] 
where $`+ $' denotes Minkowski sum.
By setting $y=0$, we get a factorization of $W_{\Lambda^{abc}_{d-1}}$, which we know is irreducible. This implies either $f$ or  $g$ is a polynomial depending only oon the variable $y$. Without loss of generality, say $g$ is a polynomial in $y$. Then $\Newt(g)$ is a line-segment along the $y$-axis. Since 
\[ \Newt(W_{\Lambda^{abc}_{d}}) = \Newt(f) + \Newt (g) \]
we have  
\[ \Newt(W_{\Lambda^{abc}_{d-1}}) \subset \Newt(f) . \] 
Proposition 4.6 of Chanda-Hirschi-Wang \cite{chw:tori} 
implies that that $\Newt(W_{\Lambda^{abc}_{k}})$ is a $k$-simplex.  This forces $\Newt(g)$ to be a point, that is, $g$ is a constant polynomial. This is a contradiction, thus $W_{\Lambda^{abc}_{d}}$ is irreducible. 
\end{proof}

 \begin{corollary}\label{lotsofleg}   The Legendrian tori, $\Lambda^{abc}_{n-1}$ and $\Lambda^{a'b'c'}_{n-1}$ in $S^{2n-1}$ corresponding to Markov triples
 $(a,b,c)$ and $(a',b',c')$ for $(a,b,c) \neq (a',b',c')$ are not Legendrian isotopic.
 \end{corollary}

\begin{proof} The argument is essentially the same as Vianna's argument that tori corresponding to distinct Markov triples are not Hamiltonian isotopic \cite{vianna:inf}.  Suppose that the Legendrians $\Lambda^{abc}_n$ and $\Lambda^{a'b'c'}_n$
are Legendrian isotopic.  After identification of first homology groups, 
we obtain an isomorphism of augmentation varieties 
\[ \Aug(\Lambda^{abc}_n)
\to \Aug(\Lambda^{a'b'c'}_n), 
\quad \Aug_\white(\Lambda^{abc}_n) \to \Aug_\white(\Lambda^{a'b'c'}_n). \] 
The corresponding augmentation polynomials are irreducible from Lemma \ref{irredofpot}. Thus, from Theorem \ref{algpotthm}, we have that the Newton polytope of $W_{\Lambda^{abc}_{n-1}}$ and $W_{\Lambda^{a'b'c'}_{n-1}}$ are the same up to change of basis. 
%
%
As explained in Vianna \cite{vianna:inf}, 
and Chanda-Hirschi-Wang \cite{chw:tori},
this equality is impossible because, for example, their Newton polygons cannot be related by any 
$GL(n-1,\Z)$-transformation, as their edges have different lattice lengths.

\end{proof}

\begin{remark}
    We could circumvent the need to prove the irreducibility of $W_{\Lambda^{abc}_{n-1}}$ if we used the equality of the augmentation schemes as discussed in Theorem \ref{augidealthm} instead of the equality of augmentation varieties.
\end{remark}

\section{Geometric augmentation varieties}

The augmentation variety has distinguished subvarieties corresponding to fillings. For a Lagrangian filling $L$ of $\Lambda$, there is an inclusion $H_2(Y,\Pi) \to H_2(\ol{X},\ol{L})$ which induces a map from $\hat{G}(\Lambda)$ to $\hat{G}(L)$. So we can change coefficients from $CE(\Lambda, \hat{G}(\Lambda))$ to $CE(\Lambda,\hat{G}(L))$. In the following, we write  $CE(\Lambda, \hat{G}(\Lambda))$ as $CE(\Lambda)$ to simplify notation.

\begin{definition} {\rm (Geometric augmentation variety)}
For any tamed  filling $L$ equipped with a bounding cochain $b \in MC(L)$ let 
\[ \varphi{(L,b)} : CE(\Lambda) \to \hat{G}(L) \]
denote the corresponding augmentation.   Let
\[ \cI_{(L,b)}(\Lambda) =  \ker(\varphi{(L,b)} ) \]
the corresponding augmentation ideal.  Let 
\[ \Aug_{L,b}(\Lambda) \subset \Aug(\Lambda) \] 
denote the variety defined by $\cI_{(L,b)}(\Lambda).$
Denote 
\[ \Aug_{\on{geom}}(\Lambda) = \bigcup_{(L,b)} \Aug_{(L,b)}(\Lambda) \] 
the union over tamed fillings with bounding chains $(L,b)$.  This ends the Definition.
\end{definition}

It will be useful to note that fillings of 
the perturbed and unperturbed Legendrians define augmentations of the same Chekanov-Eliashberg algebra. 

\begin{lemma} Let $\Lambda_\eps$ be a family of Legendrians in a stable Hamiltonian  manifold $(Z,\alpha_\eps, \omega)$, with $\alpha_\alpha$ converging to a contact
one-form $\alpha$ with corresponding two-form $ \omega = - \d \alpha$.    Then the Chekanov-Eliashberg algebras $CE(\Lambda_\eps)$ and $CE(\Lambda)$ are quasi-isomorphic, and equal for suitable choices of almost complex structure and perturbation data. 
In particular, any augmentation of $CE(\Lambda_\eps)$ 
defines an augmentation of $CE(\Lambda)$.
\end{lemma}

\begin{proof}
Indeed, for $\eps$ sufficiently small, the moduli spaces $\M(\Lambda)$
and $\M(\Lambda_\eps)$ are in bijection assuming the same almost complex structure and Morse function is used for both Legendrians.   The 
Chekanov-Eliashberg algebra $CE(\Lambda_\eps)$ has completed
coefficient ring $\hat{G}(\Lambda_\eps)$ which is a completion of 
$G(\Lambda) \cong G(\Lambda_\eps)$.   Therefore, we have a chain map 
corresponding to the completion $CE(\Lambda) \to CE(\Lambda_\eps) .$
\end{proof}

We compute the augmentation variety corresponding to the Harvey-Lawson filling.
Recall from I-\eqref{I-hl} that the Harvey-Lawson filling is 
\[ L_\eps = \{ |z_1|^2 = |z_2|^2 + \eps = |z_3|^2 + \eps , \quad 
z_1 z_2 z_3 \in (0,\infty) \}   \]
and fills the Clifford Legendrian $\Lambda_\eps = L_\eps \cap S^{5} \cong T^{2}$.

\begin{proposition} \label{augcliff} The geometric augmentation variety 
$\Aug_{\on{geom}}(\Lambda)$ of the Clifford Legendrian $\Lambda \subset S^5$ equipped with the spin structure corresponding to the element $(0,1) \in H_1(L_\eps,\Z_2)$ and trivial chain $b = 0 $ is
\[ \Aug_{(L_\eps,0)}(\Lambda) = \Aug(\Lambda)  = \{ 1 + y_1 - y_2 \} .\]
\end{proposition}

The proof is by a series of Lemmas.  We first compute the images of the degree one Morse generators.

\begin{lemma} \label{morseaug} The augmentation $ \varphi{(L_\eps,0)}  $ for the Harvey-Lawson filling $L_\eps \subset \C^5$ satisfies 
\begin{equation} \label{dd2}  \varphi{(L_\eps,0)} (\cc_1) = 0 , \quad \varphi{(L_\eps,0)} (\cc_2 ) = \ln (1 +  [\mu_1])
. \end{equation}
On any of the other $\cc_3,\cc_4$  in $\cI_\black(\Lambda)$, the augmentation 
$\varphi(L_\eps,0) $ vanishes for reasons of degree.  
\end{lemma}

\begin{proof}   First we note that the images of the Morse generators
under the augmentation must satisfy a relation. By definition of augmentation,
\[ \varphi{(L_\eps,0)} ( \delta^{\ab}(\aa)) = \varphi{(L_\eps,0)} ( 
1+  
[\mu_1] \exp(\cc_1) - 
[\mu_2] \exp(\cc_2) ) = 0  .\]
Here the sign of the $[\mu_1]$-term is positive, matching the sign of the Maslov index two
disk in \cite{chooh:fano}, while the sign of the $[\mu_2]$-term is negative since  the change in spin structure reverses the orientation of the moduli space of disks, by the discussion in \cite[8.1.2]{fooo:part2}.

Now we claim that the image of the first generator vanishes.  
The image of each $\cc_i$ under $\varphi(L_\eps,0)$ is a count of disks 
in $L_\eps$ with an incoming trajectory of $f_\black$ limiting to $\cc_i$.
We may assume that 
\[ f_{L_\eps}: L_\eps \cong S^1 \times \R^2 \to \R \] 
is the sum of the standard
height function on $S^1$ and a quadratic function on $\R^2$, and after a generic perturbation 
the images of the holomorphic disks are disjoint from $S^1 \times \{ 0 \}$.  Thus, if $u_v: S_v \to \C^3$
is a disk connecting to a trajectory $u_e: T_e \to L_\eps$, the limiting point of the 
trajectory at infinity is the limit of the flow of the projection of $u_e(T_e \cap S_v)$
under the map  $S^1 \times \R^2 - \{ 0 \} \to S^1 \times S^1$ induced by the gradient flow.  
Thus,$u_e(T_e \cap S_v)$ must map to the inverse image of the unstable manifold of $\cc_i$
under the projection from $S^1 \times \R^2$.   We may assume that the unstable manifold for $\cc_i$
is a small translate of $S^1 \times \{ 1 \}$ in the second coordinate.  Then 
there are no possibilities for $u_e(T_e \cap S_v)$ if $i = 1$, corresponding to the self-intersections
of the cycle $S^1 \times \{ 1 \}$ and exactly 
one possibility if $i = 2$, corresponding to the intersections of the cycles
$S^1 \times \{ 1 \} $ and $\{1 \} \times S^1$.  Thus
\[\varphi{(L_\eps,0)} (\cc_1) = 0 .\] 

The relation in the first paragraph now determines the value of the 
augmentation on the second generator. We have from the fact that $\varphi{(L_\eps,0)} \delta(\aa) = 0 $ that
\begin{equation} \label{dd22} \varphi{(L_\eps,0)} (\cc_2 ) = \ln ( 1 +  [\mu_1]
)  \end{equation}
as claimed.
\end{proof}

\begin{remark}  The formula for the image of $\cc_2$ is compatible with the localization computations in  Ooguri-Vafa \cite{ooguri}, Lin \cite[Section 5.7]{lin:open}, and Pandharipande-Solomon-Walcher \cite{psw:disk}. 
Indeed, the fixed point sets of the natural torus action are the multiple coveres $u_{(d)}(z) = z^d$ which count towards the computation with a  factor $(-1)^{d-1}/d$ by the localization computation similar to one given there.  More generally, in arbitrary dimension
we partially compute the augmentation associated
to the Harvey-Lawson filling in higher dimensions using localization. 
As in \eqref{dd2}, denote the standard Morse function
\[ f_\black: \Lambda \cong (S^1)^{n-1} \to \R \] 
given by the sum of the heights on each factor.  Let $\cc_1,\ldots, \cc_{n-1} \in CE(\Lambda)$ denote the associated degree 
one generators.    The
augmentation ${\varphi}$ defined by the Harvey-Lawson filling $L_{(1)}$ has 
\[ \begin{array}{ll}
 \varphi(\cc_i) = 0 &  i < {n-1} \\
\varphi( \cc_{n-1}) = 
\ln( 1 +  [\mu_1] -  \ldots +
     [\mu_{n-2}])  &  \\
 \varphi({[\mu_i]}) = {[\mu_i]} &  i < n-1 \\
\varphi([\mu_{n-1}]) =0 & \end{array}  .\] 
Indeed, let $\M(L_{(3)})$ be the moduli space of disks bounding the
  three-dimensional Harvey-Lawson filling $L_{(3)} \subset \C^3$. 
  There is a finite-to-one map of moduli spaces
\[ \phi: \M_{(n)} \to \M_{(3)}, \quad  u = (u_1,\ldots, u_{n-2},u_{n-1},u_n) \mapsto \left( \prod_{j=1}^{n-2} u_j, u_{n-1}, u_n \right) .\]
whose generic fiber has order
\[ \deg(\phi) = \left( \begin{array}{c} d \\   d_1 \ldots d_{n-2} 
                            \end{array} \right)    .
                            .\]
                            We have 
\begin{equation} \label{order} \# \phi^{-1}(u) = \left( \begin{array}{l} d \\ {} d_1 \ldots d_{n-2} \end{array} \right) .\end{equation}
By the computation of the multiple cover contributions
in \ref{hl3} and \eqref{order},
\begin{eqnarray*} \varphi( \cc_{n-1}) &=& \sum_{d_1,\ldots,d_{n-2}}
\left( \begin{array}{c} d  \\ d_1 \ldots d_{n-2}  \end{array} \right) 
\frac{ 
(-1)^{d_1 + \ldots + d_{n-2} - 1}
[\mu_1]^{d_1}\ldots 
  [\mu_{n-2}]^{d_{n-2}}}{ (d_1 + \ldots + d_{n-2}) } \\
 &=&   \ln( 1 + [\mu_1] + \ldots +
     [\mu_{n-2}]) . \end{eqnarray*}
\end{remark}

\begin{lemma} \label{projaug}
    The projection of $\Aug_{(L_\eps,0)}(\Lambda)$ onto the first factor in $\Rep(\Lambda) \cong (\C^\times)^2$
    is surjective.
\end{lemma}
\begin{proof}  We claim that there is no polynomial in the first coordinate that vanishes on the augmentation variety.  Since 
\[ \varphi{(L_\eps,0)}([\mu_1]) = [\mu_1], \quad \varphi{(L_\eps,0)}(\cc_1)
= 0, \]
any polynomial in $w_i$ maps to the corresponding polynomial in $[\mu_i]$ under $\varphi{(L_\eps,0)}$.
So there is no polynomial in $w_i$ that vanishes under $ \varphi{(L_\eps,0)}$.  
Thus, the projection of the augmentation variety on the first factor  of 
$(\C^\times)^2$ is surjective. 
\end{proof}

\begin{proof}[Proof of Proposition \ref{augcliff}]
By Lemma \ref{projaug}, $\Aug(\Lambda)$ 
is a curve and so the containment in the zero locus of the potential 
is an equality:
\[ \Aug_{(L,b)}(\Lambda) = \{ 1 +  y_1 -  y_2  = 0 \} .\]
\end{proof} 

\begin{lemma}  The geometric augmentation variety of the Clifford Legendrian $\Lambda \cong T^{n-1} \subset S^{2n-1}$ with the trivial spin structure is empty.
    \end{lemma}

    \begin{proof}
        Any geometric augmentation $\varphi(L,b): CE(\Lambda) \to \hat{G}(L)$
        must satisfy the relation coming from the differential of the minimal
        length Reeb chord
\begin{equation} \label{diffreeb2} \varphi(L,b) (W_\Lambda( [\mu_1 ]\exp(\cc_1),\ldots, [\mu_k] (\exp(\cc_k)))  )  = 0 \end{equation}
using II-\eqref{II-diffreeb}. On the other hand, the minimal area terms in  $(\Rep(p))_* W $ with respect to the filtration in II-\eqref{II-filt} all have positive coefficient one. Indeed, the areas of the corresponding disks in $Y$ are equal, by monotonicity. Thus, contributions to the images under the augmentation of the  classical generators 
\[ \varphi(\cc_i), i = 1,\ldots, n-1 \] 
must involve either positive area disks in $L$ or  insertions from the bounding chain $b$. 
        The latter contributions lie in a non-trivial subspace in the energy filtration of $CE(\Lambda)$ by assumption on the positive $q$-valuation of the bounding chain $b$.  This contradicts \eqref{diffreeb2}.
            \end{proof}

\section{Toric examples}

In this section, we make various computations and in particular, prove our main result  Theorem \ref{augpot}. 

\begin{example} \label{hl3}
We first continue the study of the Harvey-Lawson Legendrian in Lemma II-\ref{II-lem:hl2}. The perturbed Clifford Legendrian $\Lambda_\eps$ is Legendrian 
for a contact structure $\alpha_\eps$ that is a perturbation of the standard contact structure, and fibers over a non-monotone Lagrangian torus orbit
in $\CP^{n-1}$.   By Gray stability,  the contact manifold $(Z,\alpha)$ is contactomorphic 
to $(Z,\alpha_\eps)$.  The Legendrian dga's
$CF(\Lambda_0)$ and $CF(\Lambda_\eps)$ are therefore chain homotopic, 
and may be taken to be equal.  
The Harvey-Lawson filling $L_{(1)}$ defines an augmentation
\[ \varphi: CE(\Lambda_\eps) \to  G(\varphi):= G(L_{(1)})   \] 
given by (with signs depending on a suitable choice of spin structure)
\[ \begin{array}{ll} 
\varphi([\mu_i]) = [\mu_i], \   i < n-1 &
 \varphi([\mu_{n-1}]) = 1    \\
 \varphi(\cc_i) = 0, \ 
i < n-1 &  \varphi(\cc_{n-1}) = \ln( 1 +
\pm [\mu_1 ] + \ldots +
\pm [\mu_{n-2}]) .  
\end{array}\]
We obtain an augmentation 
\[ \varphi: CE(\Lambda_{\Cliff}) \to G (\varphi) \]
with the property that 
for any 
polynomial $f(y_1,\ldots, y_{n-2})$ , 
the image is 
\[ \varphi(f(y_1,\ldots, y_{n-2})) = f([\mu_1],\ldots, [\mu_{n-2}]) . \]
It follows that 
\begin{equation} \label{include} \Aug(\Lambda) \subset \left\{ 1 - \sum_{i=1}^{n-1}
  \pm y_i = 0 \right\} \end{equation}
  is a hypersurface.   Since this hypersurface is irreducible, 
the inclusion in \eqref{include} is an equality.   This ends
the Example. 
\end{example}

We describe the augmentation variety of a toric Legendrian, as claimed in Theorem \ref{augpot} from the introduction. 

\begin{theorem} \label{geompotthm}
Let $Z$ be a negative circle bundle over a Fano  toric variety $Y$ with 
minimal Chern number at least two  and so that the variety 
$ W_\Lambda^{-1}( 0)   $ is reduced.   
If the spin structure  extends over the filling constructed in Theorem 
II-\ref{II-hltoricfano}, then the augmentation variety is equal to the geometric
augmentation variety.
\end{theorem}

\begin{proof}   Since the augmentation variety 
is an irreducible hypersurface, it suffices to show that the locus defined by the given filling is also a hypersurface.  Let $\varphi{(L,0)}$ be the augmentation for the filling constructed in Theorem II-\ref{II-hltoricfano}.   We have 
\[ \varphi( y_i )  = [\mu_i], \quad i = 1,\ldots,n-2 .\]
Indeed, as in the proof of Theorem II-\ref{II-hltoricfano}, the holomorphic disks have boundary
contained in the locus 
\[\{ z_{n-2} = z_{n-1} = 0  \} \cap L \cong  T^{n-2}  .\] 
For the standard metric and Morse
function, the set of points connecting to $\cc_k$ by an infinite trajectory
$u_e: T_e \to L$ is  the  dimension two cycle obtained as the closure 
\[ \ol{A}_k \subset L, \quad
 A_k = \{ (1, \ldots, e^{i \theta}, 1, \ldots, 1 )  \} \times \R_{> 0 } .\] 
Here we hae identified 
\[ T^{n-1} \times \R_{> 0}
\subset T^{n-2} \times \R^2 \]  
by the map $e^{i \theta},r \mapsto r e^{i \theta} $
on the last two factors and the identity map on $T^{n-2}$.
For generic perturbations of these cycles, the holomorphic disks are disjoint from  $\ol{A}_k$
unless $k = n-2.$ 

We claim that the projection of the augmentation variety onto the torus corresponding to the first $n-2$ coordinates
is surjective.   Since 
\[ \varphi(y_i) = [\mu_k], \quad k = 1,\ldots, n-2 \] 
each polynomial in $y_1,\ldots,  y_{n-2}$ is mapped by $\varphi$ to the corresponding polynomial in 
$[\mu_1],\ldots, [\mu_{n-1}]$.  So there are no polynomials in 
$y_1,\ldots, 
y_{n-2}$  vanishing on $\Aug(\Lambda)$.  Since the polynomial  in \eqref{augconstraint} vanishes on $\Aug(\Lambda)$ and is reduced, equality holds.
\end{proof}

\section{Exact augmentations}

In this section we investigate augmentations associated to 
exact fillings and several examples. 

\begin{proposition}  Let $(Z,\Lambda)$ be a tame pair.  Any exact filling $(X,L)$ with trivial bounding 
chain, $b = 0$, defines an augmentation $\varphi(L,b)$ that maps
all Morse-degree-one generators $\cc \in \cI_\black(\Lambda)$
to zero. 
\end{proposition}

\begin{proof}  Since $L$ bounds  no non-constant holomorphic disks, 
$b = 0$ is a Maurer-Cartan solution.  Let $\cc \in \cI_\black(\Lambda)$
have Morse degree one.  The image $\varphi(L,b)(\cc)$ is a count of holomorphic disks in $L$ with no punctures.  Any such disk is necessarily 
constant, and the stability condition of at least three special points 
on any such disk implies that such configurations do not exist. 
Hence $\varphi(L,b)(\cc) = 0.$
\end{proof}

\begin{definition} Let $(Z,\Lambda)$ be a tame pair.  An augmentation $\varphi: CE(\Lambda) \to G(\varphi)$
is {\em exact}  if $\varphi(\cc) = 0$ for all classical generators
$\cc \in \cI_\black(\Gamma)$.
\end{definition}

\begin{example} \label{coveringex2} 
We consider the Legendrian lift of the product of equators
in the product of two-spheres.  That is, let $\Lambda \cong T^2$
denote the Legendrian 
in the unit hyperplane bundle  
\[ Z = S^3 \times_{S^1} S^3 \cong S^3 \times S^2 \] 
given by lifting the product  $\Pi = S^1 \times S^1$ in  $Y = \CP^1 \times \CP^1$.  
The disks  lifting the Maslov
index two disks in $Y$  give rise to the leading order terms in the
differential as in II-\eqref{II-deltaaa2}
\[ \delta^{\ab,0}(\aa) =
 1 -  y_1 -  y_2 + y_1 y_2    . \] 
 In this example, one sees that the augmentation variety 
 has two components defined by $y_1 - 1$ and $y_2 -1 $ respectively. 
 There is an obvious filling which produces these components:
 Write $Z$ as the unit sphere bundle in $T^* S^3$, and
 let $\Lambda$ be the unit conormal bundle of the unknot $S^1 \subset S^3$.
 Then the unit disk provides an exact filling.   Taking the two different
 ways of writing $Z$ as the unit sphere bundle gives the two 
 components of the augmentation variety, so that $\Aug_{\on{geom}}(\Lambda) = 
 \Aug(\Lambda)$.   On the other hand, one may also take $Z$ to be the unit anti-canonical bundle, 
in which case the projection from $\Lambda$ to $\Pi$ is an isomorphism obtains using the capping path  corresponding to the monomial $y_1$ 
for non-trivial spin structures on both factors is 
\[ \delta^{\ab,0}(\aa) = 1 - y_1^2 + y_1 y_2 - y_1/ y_2 . \]
The augmentation variety in this case is 
\[ \Aug(\Lambda) = \{ 1 - y_1^2 + y_1 y_2 - y_1/ y_2 = 0 \} \]
by Theorem \ref{augpot}.  

\vskip .1in \noindent {\em Claim:  There is no augmentation 
defined over a polynomial ring of the form
\begin{equation} \label{gl} \hat{G}(L) = \hat{G}(\Lambda)/ (a_1 [\mu_1] - a_2 [\mu_2]) 
\end{equation}
for some constants $a_1,a_2 \in \Z$ which maps both $\cc_1$ and $\cc_2$ to zero. }  Indeed, any such augmentation 
would map 
\[ W_\Lambda([\mu_1],[\mu_2]) := 1 - [\mu_1]^2 + [\mu_1] [\mu_2] - [\mu_1]/[\mu_2]  \] 
to zero.  This would imply that $W_\Lambda([\mu_1],[\mu_2])$ has a linear factor,
which contradicts irreducibility.

To show that $\Lambda$ has no exact filling, suppose that $L$ is a filling
of $\Lambda$. The inclusion $H_1(\Lambda) \to H_1(L)$ has one-dimensional 
kernel, since the image of $H^1(L) \to H^1(\Lambda)$ is Lagrangian (See 
\cite[Lemma 3.2.4]{fielda}.)  Hence $\hat{G}(L)$ is of the form \eqref{gl},
and the claim above shows that such an augmentation does not exist.
 \end{example}

\begin{example} \label{hopf5} 
We continue Example II-\ref{II-hopf4} regarding the exact filling $L_{(2)}$ of the 
disconnected Hopf Legendrian $\Lambda_{\Hopf} \subset S^{2n-1}$.
Depending on whether the path defining
the filling $L_{(2)}$ passes below or above the critical value in the Lefschetz fibration, 
the augmentation $\varphi(\cc_{12})$ resp. $\varphi(\cc_{21})$ has either $1$ or $n$ terms. 
\end{example}

\section{Fillings and partitions}

We compute some examples of augmentation varieties in the case
that the Legendrian is disconnected.  A conjecture
of Aganagic-Ekholm-Ng-Vafa \cite{vafaetal} describes the augmentation varieties
of Legendrians associated to links in terms of certain partitions; here we obtain a simplified
version of this conjecture, albeit in arbitrary dimension, for unions of translations of Legendrian lifts
of Lagrangian tori in toric varieties.    

We begin with the example of the {\em Hopf Legendrian}
from \cite[Equation \eqref{II-hopf}]{BCSW2} which is a disjoint of union of two copies of the Clifford Legendrian, related
by a translation by a small angle using the circle action.

\begin{lemma}
Suppose $\Lambda \cong T^{n-1} \sqcup T^{n-1} \subset S^{2n-2}$ is the Hopf Legendrian in II-\eqref{II-hopf}.  If the spin structures on the two components are identical (under the isomorphism provided by translation)
then the augmentation variety $\Aug(\Lambda)$  is a union of two irreducible components  
\begin{multline} \label{ghopf}
\Aug(\Lambda) = \{ \pm  1 \pm  y_{1,b} \pm \ldots \pm y_{n-1,b} = 0, \quad b =
1,2 \} \\ \cup \{ y_{i,1} = y_{i,2}, \quad i = 1,\ldots,
n-1 \} . \end{multline}
If the spin structures on both components of $\Lambda$ are isomorphic via the identification of components by translation and 
non-trivial, then the geometric augmentation variety is equal 
to the algebraic augmentation variety:
\[ \Aug_{\on{geom}}(\Lambda) = \Aug(\Lambda) . \]
Otherwise, if the spin structure is trivial, then the geometric
augmentation variety is empty.   

The lch spectrum $\Aug_\white(\Lambda)$ from \eqref{spectrum} is 
\begin{multline} \Aug_\white(\Lambda) = 
\{ 1 - [\mu_{1,1}] - \ldots - [\mu_{n-1,1}]   + \aa_{12}  \aa_{21} = 
1 - [\mu_{1,2}] - \ldots - [\mu_{n-1,2}]  + \aa_{21} \aa_{12} = 0 \\  
\aa_{12} ([\mu_{i,1}]  - [\mu_{i,2}])  =  0, \forall i, \quad 
\aa_{21} ([\mu_{i,1}] - [\mu_{i,2}]) = 0, \forall i
\} 
\subset \C^{2n} .\end{multline}
\end{lemma}
\begin{proof}
  We break down the possibilities for augmentations based on their values on the Reeb chords
  connecting the two components of the Legendrian. 
  With notation from Example II-\ref{II-hopf3}, suppose that an augmentation $\varphi$
satisfies either 
\begin{equation} \label{first}
\varphi(\aa_{12}) = 0 \quad \text{or} \quad \varphi(\aa_{21}) = 0 .\end{equation}
Then
\[ \varphi^{\ab}(1 \pm y_{1,b}  \pm \ldots \pm y_{n,b} ) = 0,
\quad b = 1,2 . \]
In the non-vanishing case that \eqref{first} does not hold, 
\[ \varphi(\cc_{12}) = \varphi(\cc_{21}) = 0  \] 
using II-\eqref{II-deltacc2}.  Hence
\[ 
\varphi^{\ab}( 1 - y_{1,k} y_{2,k}^{-1}) =
 0 , \quad k = 1,\ldots, n . \]
Thus, we see that the augmentation variety $\Aug(\Lambda)$ is contained
in the union of two irreducible components in the statement of the Lemma.
 On the other hand, we may define an augmentation over $G(\varphi) = \hat{G}(\Lambda)$ 
 by setting 
\[ \varphi(\cc_{i,1}) = \varphi(\cc_{i,2}) = 0 , \quad i = 1,\ldots, n - 1 \]
and choose 
\[ \varphi(\aa_{12}),  \varphi(\aa_{21}) \]
 so that 
 \[ 1 \pm [\mu_{1,k}] e^{\varphi(\cc_{1,k}) } \pm 
 \ldots \pm [\mu_{n-1,k}] e^{\varphi(\cc_{n-1,k})}  + \varphi(\aa_{12}) \varphi(\aa_{21})  = 0  . \]
 Then $\varphi$ defines an augmentation.

 Suppose the spin structures on the two connected components are non-trivial and isomorphic; we will show that the components of the augmentation variety are geometrically realizable.    The first component in \eqref{ghopf} is geometrically realized by the union of the two Harvey-Lawson fillings as in the proof of Theorem \ref{geompotthm}.  We claim that the second component is that of the filling constructed in II-\eqref{II-L2}.  Since the filling is exact, the
  augmentation vanishes on the generators corresponding to the critical points of $f_L$:
\[  \varphi(\cc_i^a) = 0, \quad a \in \{1,2 \} .\]
  On the other hand, the map $H_1(\Lambda) \to H_1(L)$ identifies
  $\mu_i^1$ with $\mu_i^2$ for each $i$ and so
\[  \varphi( [\mu_i^1] ) = \varphi( [\mu_i^1] \exp(\cc_{i,1}) ) = 
 \varphi( [\mu_i^2] ) = \varphi( [\mu_i^2] \exp(\cc_{i,2}) ) . \] 
The augmentation component corresponding to this filling  lies in the diagonal component
of the augmentation variety is therefore
\[ \Aug(\Lambda)_L =  \{ (y_1,y_2) \in \cR(\Lambda_1) \times
\cR(\Lambda_2), \quad y_1 = y_2 \} . \]
The lch spectrum is 
defined by the equations corresponding to the differentials
of the degree one generators.  These generators
are $\aa ,\cc_{1,1}, \ldots,  \cc_{n-1,1},\cc_{1,2},\ldots, \cc_{n-1,2}$
as in Lemma II-\ref{II-deltacc2},
where the differentials were computed.
It follows that the 
lch spectrum  is as stated. 
\end{proof}

The Lemma above gives an example of a Legendrian
where the augmentation variety and the lch spectrum disagree. 

 We compute the geometric augmentation variety of a Legendrian with three components 
and show that the augmentation variety has four irreducible components. 
Consider the Legendrian
\[ \Lambda = \Lambda_0 \cup \beta \Lambda_0 \cup \beta^2
\Lambda_0  \subset S^{2n-1} \] 
  where 
\[ \beta = \exp ( i \theta) \]  
is a root of unity with $\theta > 0$ small. 
For any partition $ \{ \{ i, j \}, \{ k \} \} $ of $\{ 1,2, 3 \}$
define 
Then
\[ \Aug(\Lambda)_{\{ \{ i, j \}, \{ k \} \} } = \{  y_{i,b} = y_{j,b}, \quad \forall b, 
\pm 1 \pm y_{k,1} \pm y_{k,2} = 0  \} \] 
with signs depending on the choice of relative spin structure.
On the other hand, define 
\[ \Aug(\Lambda)_{\{ \{  1 \}, \{ 2 \}, \{ 3 \} \}} = 
 \{  
\pm 1 \pm y_{k,1} \pm y_{k,2} = 0 , k  \in \{ 1, 2,  3 \}  .\} \] 

\begin{theorem} \label{threecomp} Let $\Lambda \subset S^{2n-1}$
denote the three-component Legendrian described above.  For any 
spin structure on $\Lambda$ so that the spin structures on each component are isomorphic via translation, the augmentation variety $\Aug(\Lambda)$
is the union 
\begin{multline} \Aug(\Lambda ) = 
\Aug(\Lambda)_{\{ \{  1 \}, \{ 2 \}, \{ 3 \} \}} \cup \Aug(\Lambda)_{\{ \{  1 , 2 \}, \{ 3 \} \}} \\
\cup \Aug(\Lambda)_{\{ \{  1 \}, \{ 2 , 3 \} \}} \cup \Aug(\Lambda)_{\{ \{  1 ,3 \}, \{ 2 \} \}} . \end{multline}
If the spin structure is non-trivial, then the augmentation variety $\Aug(\Lambda)$
is equal to the geometric augmentation variety $\Aug_{\on{geom}}(\Lambda)$, and each irreducible component corresponds to a filling. 
    \end{theorem}
    
Thus, in total there are four components
of the augmentation variety, corresponding to partitions of
$ \{ 1,2, 3 \}$ into subsets of size at most $2$.  

\begin{proof} We introduce notation for the generators of the Chekanov-Eliashberg algebra. Let 
\[ \Lambda_j = \exp ( j i \theta) \Lambda_{\Cliff} \] 
be the $j$-th sheet of $\Lambda$ and $[\mu_{b,j}], b = 1,2$ 
the standard basis for $H_1(\Lambda_j)$ and
\[ y_{j,b} = [\mu_{j,b}] \exp(\cc_{i,b})  \in CE(\Lambda), \quad b = 1,2, \  j = 1,2,3.\]  
Let $\aa_{jj}$ denote the Reeb chords of minimal length connection 
each sheet to itself, of real degree one, over the degree zero critical point in $\cR(\Lambda)$
and let  $\aa_{jk} $  be the Reeb chords for
$j < k$ connecting the $j$-th sheet to the $k$-th sheet.  The real degree 
is
\[  \deg_\R( \aa_{jk}) =  \frac{3}{\pi} (k - j)\frac{2\pi}{9}  - 1  = \frac{2}{3}(k - j) - 1. \] 
In particular, 
\[ \deg_\R(\aa_{12}) = \deg_\R(\aa_{23}) = \deg_\R( \aa_{31}) = -\frac{1}{3} \] 
while 
\[ \deg_\R(\aa_{13}) = \deg_\R(\aa_{21}) = \deg_\R( \aa_{32}) = \frac{1}{3} .\] 
The $\Z_2$ grading of all these elements would be $1$ mod $2$ so that 
\[ \deg_{\Z_2} (\aa_{ik}) = \deg_{\Z_2}(\aa_{ik})  = 1, \quad 
\deg_{\Z_2}(\aa_{ij}) =  \deg_{\Z_2}(\aa_{jk}) = 1  . \] 

We may compute the differential of some of the generators. 
As in II-\eqref{II-deltacc}, but now with an additional component, we have 
\[ \delta^{\ab,0}(\aa_{ii}) = \pm 1 \pm  y_{i,1} \pm  y_{i,2} + 
\sum_{k \neq i} \aa_{ik} \otimes \aa_{ki} .\]
On the other hand, for the real-degree $\frac{1}{3}$ generators we have  
\[ \delta^{\ab,0}(\aa_{ik}) = \aa_{ij} \otimes \aa_{jk} , \quad \forall (i, j
, k ) \in \{ (1,2,3), (2,3,1), (3,1,2) \} \]
arising from the cover of the constant disk in $\C P^2$, with no other
outputs possible since any outgoing collection of Reeb chords with
smaller total angle change for a punctured disk would force the
projection into $Y$ of the disk to be non-constant, in which case the
outgoing angle change would be negative.  Thus, we have for any
augmentation $\varphi$
\[ \varphi(\aa_{ij}) \otimes \varphi(\aa_{ji}) \neq 0 \quad 
\implies \quad \varphi(\cc_{ij}) = \varphi(\cc_{ji}) = 0  \] 
in which case $y_{i,k} = y_{j,k}$ for all $k$.  Also, since $\varphi$
is an augmentation,
\[ \varphi(\delta^{\ab,0}(\aa_{ik})) = \varphi(\aa_{ij}) \otimes \varphi(\aa_{jk})  = 0
\]
so either $\varphi(\aa_{ij})$ or $\varphi(\aa_{jk})$ vanishes.
Without loss of generality assume that $\varphi(\aa_{jk})$ vanishes.
Then 
\[ \varphi(  \pm 1 \pm  y_{i,1} \pm  y_{i,2} + \aa_{ij} \otimes \aa_{ji}) =0  .\]
If $\varphi(c_{ij})$ also vanishes then
\[  \pm 1 \pm  y_{k,1} \pm  \ldots \pm y_{k,2}  = 0 , k  \in \{ 1,2,3 \}  . \]
If $\varphi(c_{ij})$  is non-zero then the relation
\[ y_{b,i} = y_{b,j}, \quad b \in \{ 1,  2 \}  \]
must hold.  Putting everything together, $\Aug(\Lambda)$ is contained in the union of the four irreducible
components in Theorem \ref{threecomp}. 

On the other hand, we may construct augmentations as follows. 
For the partition $\{ \{ 1 \}, \{ 2 \}, \{ 3 \} \}$ we may choose
$\varphi(\cc_{j,b})$ so that 
\[  \varphi(\cc_{j,1}) = 0, \quad  \pm \varphi(\cc_{j,2})  = \ln(1 \pm [\mu_{j,1}]) \]
for all $j$, and 
\[ \varphi(\aa_{12}) = \varphi(\aa_{21}) = 0. \]
For the partition $\{ \{ 1, 2 \}, \{ 3 \} \}$, for any 
values $\varphi(\cc_{j,1}) = \varphi(\cc_{j,2}) $, for $b \in \{ 1,2  \}$, choose
$\varphi(\aa_{12}), \varphi(\aa_{21})$ so that 
\[ 1 \pm  \varphi(\cc_{1,1}) \pm \varphi(\cc_{1,2}) + \varphi(\aa_{12}) 
\varphi(\aa_{21}) = 0 . \] 
Choose $\varphi(\cc_{j,3})$ so that 
\[ 1 \pm \varphi(\cc_{3,1}) \pm \varphi(\cc_{3,2}) = 0\]
We claim that $\varphi$ vanishes on the image of $\delta$. Indeed, we have
taken the standard Morse function on $\Lambda$, so  for each Morse-degree-two generator $\bb \in \cI_\white(\Lambda)$,  the image $\delta(\bb)$ is the Morse differential, which vanishes, 
plus terms that have at least one skew-symmetry as in Proposition II-\ref{II-abelian}. Thus, $\varphi(\delta(\bb)) = 0.$  Similarly, if $\aa$ is a minimal length Reeb chord representing a degree-zero generator
in $\cI(\white)$ then the formula \eqref{aarel} implies that $\varphi(\delta(\aa)) = 0$.
Thus,$\varphi$ vanishes on $\on{im}(\delta)$ and defines an augmentation.

For non-trivial spin structures. one has a symplectic filling by taking the union of the filling of II-\eqref{II-L2} and a copy of the Harvey-Lawson filling.   One also has a filling
given by the union of three translates of the Harvey-Lawson filling.
The holomorphic disks were computed in Example \ref{hopf5}.
Taking these together one sees that each component in 
the statement of the Theorem \ref{threecomp} actually appears.
\end{proof} 

A similar computation holds for any union of translates of the Clifford
Legendrian:

  \begin{theorem} \label{partitions}
  For a disjoint union $\Lambda \cong T_{\Cliff}^{\sqcup \ell}$ of $ \ell$ translated copies of the Clifford Legendrian, the augmentation variety   $\Aug(\Lambda)$ is the union of irreducible components $\Aug_P(\Lambda)$
  indexed by partitions $P$ of $\{ 1, \ldots, \ell \}$ into subsets of size one or two so that the spin structures on the 
  components $\Lambda_j,\Lambda_k$ agree for each 
  pair $\{ j, k \} \in P$.    Each component $Aug_P(\Lambda)$ is the product of products of $\Rep(p)( W^{-1}_\Pi(0))$ for each singleton in the partition $P$, assuming irreducibility,
  with varieties of the form 
\[ \{  y_i^{a} = y_i^b \}  \]  
for each pair of integers $a,b$ such that $\{ a, b \}$ lies in the partition $P$.  If the spin structures in the singletons in $P$ are non-trivial, then 
$\Aug_P(\Lambda)$ corresponds to a geometric filling.  
\end{theorem} 

We leave the proof to the reader, as it is similar to Theorem \ref{threecomp}.    Note
that the paper \cite{kps:pre} gives an interpretation of linearized contact
homology (viewed as a ``sheaf'' over the augmentation variety) in terms of the base data. 

\section{Ruling out exact fillings}

In this section, we use the moduli spaces of buildings to rule out the existence of 
exact fillings of the Clifford Legendrian.   We sketch the argument:  Each contribution to the differential of the Chekanov-Eliashberg algebra is a lift of a Maslov index two disk, which may be viewed as a boundary point of some component the moduli space of holomorphic maps bounding the filling.   Each such component may be viewed as a relation of cycles in the filling, and we show that there are too many such relations for the cohomology of the 
filling to restrict to an isotropic subspace of the boundary. 

\begin{theorem} \label{nofill} (Dimitroglou-Rizell \cite{dr:few} and
  Treumann-Zaslow \cite{treumann:cubic} for $n=3$) For $n >2$ the
  Clifford Legendrian $T^{n-1} \subset S^{2n-1}$ has no exact Lagrangian filling.
\end{theorem}

\begin{proof}[Proof for $n$ odd]  In odd dimensions, the proof of the Theorem follows from the fact that the number of boundary components of the one-dimensional moduli space is necessarily even, which contradicts the fact that there are an odd number of Maslov index two disks bounding the projection.   In more detail, let %
\[ \iota_*:H_1(\Lambda) \to H_1(L) \] 
  denote the inclusion map.  Suppose that $L$ is an exact 
  relative spin filling.   By the results of the previous sections, $L$ defines an augmentation
  \[ {\varphi}: CF(\Lambda) \to \hat{G}(L) . \] 
  Necessarily 
  $\varphi$ vanishes on the image of the differential and so 
 \begin{eqnarray} \label{varphizero}
   0 =    {\varphi}(\delta^{\ab}(\aa)) &=& {\varphi}( \pm 1 \pm y_1
                                \pm \ldots \pm  y_n) \\
                            &=& \pm 1 \pm {\iota_* [\mu_1]} \pm \ldots \pm
                                {\iota_* [\mu_{n}]}.
\end{eqnarray}
Here $\varphi$ is defined over the group ring $G(L)$; integrating the identity \eqref{varphizero} over $L$ gives 
\[ 0 = \pm 1 \pm 1  \pm \ldots \pm 1 \]
with an odd number of terms on the right-hand side.  Since zero is even, this is a contradiction.

We generalize the argument to the case that the filling is not necessarily relatively spin 
by examining the one-dimensional component in the  moduli space of buildings directly. The one-dimensional component of the moduli space of buildings $\M(L)_1$
contains $\M(\Lambda)_0$ as a collection of boundary components, where a building $u = (u_0,u_1)$ in $L$
is obtained from a building $u_1$ in $\Lambda$ by considering the first level  $u_0$ to be empty. 
Since the Lagrangian $\Pi \subset Y$ is monotone in this case and $L$ bounds no holomorphic disks, there are no boundary components of $\M(L)_1$ other than those arising from $\M(\Lambda)_0$:
\[ \partial \M(L)_1 = \M(\Lambda)_0. \]
As a result, the number of rigid buildings in $\M(\Lambda)_0$ asymptotic to a given incoming Reeb chord
$\gamma \in \cR(\Lambda)$ must be even.  Each rigid building in $\Lambda$ with Reeb chord $\aa$ correspond to a Maslov index two disk in $\CP^{n-1}$.
The number of such is equal to $n$ and so odd, which is a contradiction. 
\end{proof} 

In the case of even dimension we examine the moduli space of buildings
more closely.    We show relations in the homology arising from the moduli spaces of buildings in both degree one and codegree one. We begin with 
the degree one relations:

\begin{lemma} \label{onerelations}
Let $L$ be a compact, oriented exact filling of a compact Legendrian $\Lambda \cong T^{n-1}$, and $[\mu_1], \ldots, [\mu_{2n-1}] \in 
H_1(T^{n-1},\Z_2)$ the standard basis in first homology.
After re-ordering, the image of $H_1(\Lambda,\Z_2)$
in $H_1(L)$ is described by the relations
\[ \iota_* [\mu_1] = 0, \quad \iota_* [\mu_{2i} ] = \iota_* [\mu_{2i+1}], \quad i =
1,\ldots, \frac{1}{2}(n-1) . \]
\end{lemma}

\begin{proof}  The statement of the Lemma follows from a matching of the ends of the one-dimensional components of the moduli space of buildings bounding the filling.
Let $\M_i(\Lambda)_0$ denote the moduli space of rigid buildings in $\R \times Z$ of class $\iota_* \mu_{i}$.   Each building in $\M_i(\Lambda)_0$ represents the
boundary of a one-dimensional component $\M(L)_1$ of the moduli space  of
buildings $u: C \to \XX$  bounding the filling $L$.  Since each such component of $\M(L)$
has two ends and any one-manifold is oriented, without loss of generality $\M_{2i}(\Lambda)$ is connected to $\M_{2i+1}(\Lambda)$ by such a component for $i \ge 1$.  After re-ordering, 
we may assume that the components 
joined by the cobordism are those containing the 
$2k+1$-st and $2k+2$-st Maslov index two disks bounding $\Pi$, for each $k = 1,\ldots, n/2 -1 $.  After capping off using a path which is the reverse
of the lift of the first Maslov-index-two disk, the homology classes of these  disks are $0, [\mu_1], \ldots, [\mu_{n-1]}$. The statement of the Lemma follows.
\end{proof}

To obtain relations in higher degree homology groups, we study more general moduli spaces of buildings as follows.  We introduce the following moduli space with constraints in a submanifold of Reeb chords. 

\begin{definition}  For a tame pair $(Z,\Lambda)$ let
$ A \subset {\cR}(\Lambda) $ be a closed submanifold of the component
${\cR}(\Lambda)_{\min} \cong \Lambda$ of ${\cR}(\Lambda)$ consisting of
Reeb chords of angle change $2\pi /n$.  Let
\[ \ol{\M}(L,A) = \{ u: S \to \XX, \ev_e(u) \in A  \} \] 
the moduli space of holomorphic maps $u: C \to X$ with domain a disk $S$ with a single strip-like end
like end $e$ asymptotic along that end to a Reeb chord in $A$, lifting the Maslov index two disks
bounding $\Pi$.
\end{definition} 

We investigate the boundary of the moduli space of once-punctured disks bounding the filling with evaluation in the given cycle on the Legendrian.  The moduli of disks $\ol{\M}(L,A)$ bounding the Lagrangian contains the subset $\M(\Lambda,A)$ of disks in $\R \times A$ bounding
$\R \times \Lambda$ with a limiting Reeb chord in $A$, considered as
buildings $(u_0,u_1)$ in $X$ with the first component $u_0$ empty.
Recall from for example Zinger \cite{zinger}:

\begin{definition} 
Let $X$ be a compact manifold.      A {\em pseudocycle-with-boundary} is a pair $(Z,\iota)$
consisting of an oriented manifold with boundary $Z$ of dimension $\dim(Z) \in \Z_{\ge 0}$ and a smooth map $\iota: Z \to X$
so that the closure $\ol{\iota(Z)}$ is equal to the union of the image $\iota(Z)$ and a finite collection of images $\iota_k(Z_k)$ of manifolds $Z_k, k =1 ,\ldots, m$ of dimension $\dim(Z_k)$ at least two less than $\dim(Z)$.
\end{definition}

\begin{lemma} Let $Z = S^{2n-1}$ and $\Lambda = T^{n-1}$ the Clifford
Legendrian.  For the standard complex structure on $\R \times Z$, the moduli space $\M(\Lambda,A)_{\dim(A)}$ 
is smooth and compact.  For generic perturbations,
  the component $\ol{\M}(L,A)_{\dim(A)+1}$ of dimension $\dim(A)+1$ is a pseudocycle-with-boundary whose boundary is diffeomorphic to $\M(\Lambda,A)_{\dim(A)}$.
\end{lemma}

\begin{proof} The regularity statement on $\M(\Lambda,A)_{\dim(A)}$ follows from the
  lifting properties in Section II-\ref{II-liftsec} and regularity of disks in 
  $\CP^{n-1}$ bounding the Clifford torus of Maslov index two; these are given 
  by Blaschke products with a single factor and so regular.  Since there is a
  single boundary-puncture going to a Reeb chord $\gamma$ with the smallest possible angle at infinity, any codimension one bubbling in
  $\M(L,A)_{\dim(A)+1}$ must involve bubbling off a positive-area disk $u_v: S_v \to X$ without ends.   Such bubbling is impossible by the exactness assumption.  Therefore, the moduli space $\M(\Lambda,A)_{\dim(A)}$
  considered as a subset of the moduli space of buildings $\M(L,A)_{\dim(A)+1}$ (with empty first level) is the only boundary stratum.   The Cieliebak-Mohnke regularization from \cite{cm:trans} equips $\M(L,A)$ with the structure of a pseudocycle; 
  the only possible bubbling (besides a level going off to infinity along the 
  cylindrical end) appears when markings constrained to map to the Donaldson hypersurface come together to form a ghost bubble with more than one marking.
  As in \cite{cm:trans}, these configurations are not cut out transversally but 
  represent a tangency with the Donaldson hypersurface.  Such tangencies
  are real codimension at least two and their image is covered by 
  configurations obtained by removing the ghost bubble and enforcing a tangency of the required order at an interior marking.  To end the argument, 
  one needs charts for $\ol{\M}(L)$ considered as a manifold with boundary
  near $\M(\Lambda)$.  These are produced by a gluing argument as in 
  \cite[Lemma 5.12]{pardon}, in which the gluing parameter (representing the 
  translation which the map to $\R \times Z$ is glued into the neck of 
  $X$) and the coordinates on $\M(\Lambda)$ locally produce coordinates
  on $\M(L)$. \end{proof} 

\begin{proposition}
Let $K$ be a compact oriented manifold with boundary $\partial K$.  The image of $H^\bullet(K)$
in $H^\bullet(\partial K)$ is a maximally isotropic subspace with respect to the Poincar\'e pairing.
In particular, for $n \ge 4$ the image of the restriction map 
\[ H^1(K) \oplus H^{n-2}(K) \to H^1(\partial K) \oplus H^{n-2}(\partial K) \]
is of dimension $\dim(H^1(\partial K))$.
\end{proposition}

\begin{proof} The first claim is \cite[Lemma 3.2.4]{fielda}, although it must be more widely known.
The second claim follows since the pairing between $H^1(\partial K)$ and $H^{n-2}(\partial K)$ is perfect.
Thus, any maximally isotropic subspace of split form is of the form $V \oplus V^{\ann}$ where $V \subset H^1(K)$
and $V^{\ann} \subset H^{n-2}(\partial K)$ is its annihilator.
\end{proof}



We obtain relations in codimension one by considering moduli spaces constrained to pass through codimension one cycles.  
Consider the Pontrjagin product
\[ H_\bullet( \Lambda,\Z_2) \times H_\bullet( \Lambda,\Z_2) \to
H_\bullet( \Lambda,\Z_2) \] 
generated by the group multiplication
$\Lambda \times \Lambda \to \Lambda$ using the diffeomorphism
$\Lambda \cong T^{n-1}$.  For each $i$ define a class $\vv{\mu}_i$ by taking the Pontrjagin product of the
classes $[\mu_j]$ with $j \neq i$;
\[ [\vv{\mu}_i] = [\mu_1][  \mu_2] \ldots [\mu_{i-1}][  \mu_{i+1}] \ldots
[\mu_{n-1} ] \in H_{n-2}(\Lambda,\Z_2) .\]
Define a {\em vanishing subspace of degree one classes } corresponding to the relations in 
Lemma \ref{onerelations}:
\begin{eqnarray*} 
I  &=& 
\on{span} (  [\mu_1], [\mu_2] - [\mu_3], [\mu_4 ]- [\mu_5], \ldots , [  \mu_{n-2 } ] -  [\mu_{n-1}]) \subset H_1(\Lambda,\Z_2) .\end{eqnarray*}
We showed in Lemma \ref{onerelations} that $I$ is contained in the kernel of the inclusion map of the boundary in homology.

\begin{proposition} Suppose that $\Lambda = T^{n-1}$ is the Clifford
Legendrian in $Z = S^{2n-1}$.  The kernel $\hat{I}$ of the inclusion map
  $H_{n-2}(\Lambda,\Z_2) \to H_{n-2}(L,\Z_2)$ has dimension 
  strictly greater than $(n-1)/2$.
\end{proposition}

\begin{proof}
Consider the moduli space
  of punctured holomorphic  disks whose boundary evaluation is constrained
  to lie in a codimension two subtorus.  We identify the component of minimal length chords with $\Lambda \subset \cR(\Lambda)$.  Let
  $A_i \subset \Lambda$ be the circle in the $i$-th component, 
\[ A_i := \{ z_j =1,  j \neq i \} \subset T^{n-1} . \]
Define the codimension two sub-torus
\begin{equation} \label{Aprod} \hat{A}_i = A_1 \ldots A_{i-1} A_{i+2} \ldots A_{n-1}  \subset \Lambda.\end{equation}
skipping the $i$ and $i+1$-st factors. 
Let $\ev_{\Lambda,1}: \M(\Lambda) \to \cR(\Lambda)$ be evaluation at the puncture, and 
\[ \M(\Lambda,\hat{A}_i) = \ev_{\Lambda,1}^{-1}(\hat{A}_i) \]
the moduli space of once-punctured disks in 
$\R \times Z$ bounding $\R \times \Lambda$ with the puncture mapping
to $\hat{A}_i$; that is, the space of lifts of Maslov index two disks in $\CP^{n-1}$ with Reeb chord starting at $\hat{A}_i$.  Since Blaschke disks of index two are regular, the space of lifts has a natural action of $\Lambda \cong T^{n-1}$.  The evaluation map is transverse to 
$\hat{A}_i$, so the moduli space 
$\M(\Lambda,\hat{A}_i)_{\dim(\hat{A}_i)}$ is smooth and diffeomorphic to 
$\hat{A}_i$ via the evaluation map since each lift is unique by Theorem \ref{II-liftthm}.

We now view the moduli space of disks in the previous paragraph as
a boundary component of moduli space of punctured disks bounding the cobordism.  Consider the moduli space $\M(L)$ of once-punctured
disks bounding $L$.   Evaluation at infinity $\ev_{L}$ still defines a map 
to $\cR(\Lambda)$, and denote 
\[ \M(L,\hat{A}_i) = \ev_{L}^{-1}(\hat{A}_i) \] 
the moduli space of once-punctured disks 
with Reeb chord in $\hat{A}_i$. 
This moduli space contains $\M(\Lambda,\hat{A}_i)$ in its boundary, 
since we may view $\M(\Lambda,\hat{A}_i)$ 
as a moduli space of two-level buildings
where the first level is empty. 
Let $K_i$ be the connected component of $\M(L,\hat{A}_i)$ containing $\M(\Lambda,\hat{A}_i)$ in its boundary.
The connected component $K_i$ may also contain possibly other lifts of Blaschke products $u_k$ for $k$ in some subset  
\[ I_i \subset \{ 1, \ldots, n \} \] 
as shown in Figure \ref{corelns}.
By Lemma \ref{onerelations}, after re-ordering we may assume that $I_{2k+1}$ contains $2k+2$ for each $k$.
 \begin{figure}[ht]
     \centering
     {\tiny 
     \scalebox{.8}{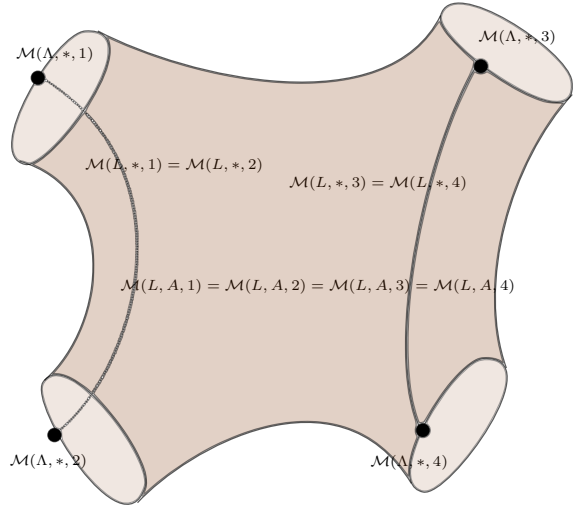}}
     \caption{Homological relations between components of moduli spaces}
     \label{corelns}
 \end{figure}

In order to construct a cobordism we consider moduli spaces with a marking 
at the boundary that is not a puncture. Consider moduli spaces $\M_1(\Lambda,\hat{A}_i)$,
$\M_1(L,\hat{A}_i)$ 
of pairs $(u,z)$ where $u: C \to \R \times Z$ or $u: C \to \R $ is a once-punctured holomorphic disk bounding $L$ or $\Lambda$ and $z \in \partial C$ is a single marking on the boundary. There are evaluation maps at the boundary
\[ \ev_{\Lambda,1}: \M_1(\Lambda,\hat{A}_i)
\to \Lambda,  \ 
\ev_{L,1}:\M_1(L,\hat{A}_i) \to L \]
defined by mapping $(u,z)$ to $u(z)$.
The forgetful maps 
\[ \phi_\Lambda: \ \M_1(\Lambda,\hat{A}_i) \to 
 \M(\Lambda,\hat{A}_i), \ 
 \phi_L: \M_1(L,\hat{A}_i) \to 
 \M(L,\hat{A}_i)
 \]
 have fiber $[-\infty,\infty]$; the endpoints correspond to configurations 
containing $(u_1,u_2)$ with $u_2: C_2 \to X$ that  a trivial strip with a marking $z
\in \partial C_2$.
Let $K_{1,i}$ be the connected component of $\M(L,A_i)$ containing $\M(\Lambda,\hat{A}_i)$ in its boundary.
No other breakings are possible, since the action of the Reeb chord at the puncture is minimal.  

The connected component in the previous paragraph has two kinds of boundary components, where either the map or the marking has gone off to infinity.
We  construct a cobordism in $\Lambda$ by capping off the Reeb chords to form cycles. 
Recall from Definition \ref{I-boundclass} the  definition of a capping path $\hat{\gamma}$
for a Reeb chord $\gamma.$
By translation, 
a capping path $\hat{\gamma}$ for one element 
of $\M(\Lambda,\hat{A}_i)$ limiting to a Reeb chord $\gamma = \ev_e(u)$ starting
at the identity induces capping paths 
\[ (\ev_e(u))(0) \hat{\gamma}: [0,1] \to \Lambda \]
for all elements in the same component by translation by $\ev_e(u)(0) \in \Lambda \cong T^{n-1}$.   Define a ``capped off moduli space'' 
\[ \widehat{\M}_1(\Lambda,\hat{A}_i) = {\M}_1(\Lambda,\hat{A}_i)
\cup \left(  {\M}(\Lambda,\hat{A}_i)  \times [0,1] \right) .
\]
A map 
\[ \phi_{\Lambda,i}: \widehat{\M}_1(\Lambda,\hat{A}_i) \to \Lambda  \] 
is defined by evaluation of the 
map on the first component
\[ {\M}_1(\Lambda,\hat{A}_i)\to \Lambda, \quad (u,z) \mapsto u(z) \] 
and evaluation of the capping path $(\ev_e(u))(0) \hat{\gamma}$ for the limit
$\ev_e(u)$ on the second:
\[ {\M}(\Lambda,\hat{A}_i) \times [0,1] \to \Lambda, 
\quad (u,t) \mapsto (\ev_e(u))(0) \hat{\gamma}(t) . \]
The same construction applies to the moduli space of punctured disks bounding the filling:  
\[ \widehat{K}_{1,i} = 
K_{1,i}  \cup( K_i \times [0,1] )  .
\]
Via evaluation the capped off moduli space gives rise to a map 
\[ \phi_{L,i}: 
\widehat{K}_{1,i} \to L .\]
with boundary $\widehat{\M}_1(\Lambda,\hat{A}_i)$, where 
$\phi_{L,i}$ is defined by evaluation 
at the marking on $K_{1,i}$, and evaluation fo the capped off path on 
$K_i \times [0,1] $.  The capped-off moduli space space is a topological manifold with boundary.   The boundary components isomorphic to 
$K_i$, consisting of configurations where the marking has bubbled off onto a constant strip, each occur twice, once in $K_{1,i}$
and once in $K_i \times [0,1]$.
This leaves the boundary component $\M(\Lambda,\hat{A}_i)$, so that 
\[  \partial 
\widehat{K}_{1,i} =  \M(\Lambda,\hat{A}_i) .\]
Push-ing forward to homology we get 
\[ \phi_{L,i,*} (\partial 
\widehat{K}_{1,i}) = \iota_* \phi_{\Lambda,i,*}( \M(\Lambda,\hat{A}_i)) \]
where $ \iota: \Lambda \to L $  is the inclusion of the boundary. 

 We compute the homology class
of the boundary.  Consider the forgetful fibration 
\[ \M_1(\Lambda,\hat{A}_i) \to \M(\Lambda,\hat{A}_i)_{\dim(A)} . \]
Each fiber evaluates to the translation of the boundary 
$\partial u_i$ by an element of $A$, the moduli space being invariant
under multiplication by $A$.   It follows that the homology class
of the image in $\Lambda = \partial L$ is
the Pontrjagin product 
\[  [\phi_{\Lambda,i}(\widehat{\M}_1(\Lambda,\hat{A}_i)] = \mu_i [A] = \mu_1 \ldots \mu_i \mu_{i+2} \ldots \mu_{n-1} \in H_{n-2}(\Lambda,\Z_2) . \]

Finally, we compute the relation in homology obtained from the filling.
We have
\begin{eqnarray*} 0 &=&
\phi_{L,i,*} [ \partial 
\widehat{K}_{1,i}] \\
&=& \iota_* \phi_{\Lambda,i,*}( \M(\Lambda,\hat{A}_i)) \\
&=& \sum_{k \in I_i} \iota_* \mu_k  \iota_* [\hat{A}_i]   \in H_{n-2}(L,\Z_2). \end{eqnarray*}
Since the Pontrjagin products $\mu_m \mu_m = 0$ vanishes for all $m$,
by definition of $\hat{A}_i$ in \eqref{Aprod}, 
\[ [\mu_k] [\hat{A}_i]  = 0 \quad \text{unless} \quad k \in \{ i,i+1 \} . \]
This gives the relation 
\[ [\vv{\mu}_{2k}] +   \vv{\mu}_{2k+1} = [\mu_{2k}] [\hat{A}_{2k}] + \mu_{2k+1} [\hat{A}_{2k+1}] = 0\]
for $k = 1,\ldots, (n-2)/2$.  On the other hand, 
taking $i = 1$ This gives either the relation 
\[ [\vv{\mu}_{1}]  = [\mu_1] [\hat{A}_1] = 0 \]
if $2 \notin I_1$
or the relation 
\[  [\vv{\mu}_{1}] + [\vv{\mu}_2]  = 0 \]
if $2 \in I_1$.  Either way, the dimension of $\hat{I}$ is greater
than $(n-1)/2.$
\end{proof}

\begin{lemma} \label{lbound} $\dim(I) + \dim(\hat{I}) > n-1$.
\end{lemma}

\begin{proof}  There are $n/2$ independent vectors in $I$
and $n/2$ in $\hat{I}$ for a total of $n$. 
\end{proof}

\begin{proof}[Proof of Theorem \ref{nofill}] 
  It remains to show Theorem \ref{nofill} in the case $n = 2k$ is even.  The dimension of 
  the subspaces $I$
  and $\hat{I}$ are both half of $\dim(\Lambda) + 1$, and so larger than  
  $ \frac{1}{2} \dim(H_1(\Lambda,\Z_2))$.  On
  the other hand, the intersection pairing
  \[ H_1(\Lambda,\Z_2) \times H_{n-2}(\Lambda,\Z_2) \to \Z_2 \]
  is trivial on $\ker(\iota_*)$.  Indeed, suppose that
  \[ C_- \subset \Lambda, \quad C_+ \subset \Lambda \] 
  are pseudocycles that extend
  to pseudocycles with boundary $\hat{C_-}, \hat{C_+}$ in $L$ bounding
  $C_-,C_+$.  The intersection $\hat{C_-} \cap \hat{C_+} \subset L$ is, after
  generic perturbation, a one-dimensional submanifold with boundary
  \[ \partial (\hat{C_-} \cap \hat{C_+}) = C_- \cap C_+ . \] 
  Thus $C_- \cap C_+$ represents the zero homology class in  $H_0(\Lambda,\Z_2)$.  Since the pairing
  $H_1(\Lambda,\Z_2) \times H_{n-2}(\Lambda,\Z_2) \to \Z_2$ is non-degenerate, the dimension of any
  isotropic subspace is at most dimension $(n-1)$. This contradicts the estimate
  $\dim(I \oplus \hat{I}) > n-1$ from Lemma \ref{lbound}.
\end{proof} 

\bibliography{leg}{}
\bibliographystyle{plain}

\end{document}